\documentclass{amsart}
\usepackage[utf8]{inputenc}
\usepackage{enumitem}
\usepackage{amssymb}
\usepackage{comment}
\usepackage{cleveref}
\usepackage{thmtools}
\usepackage{mathrsfs}
\usepackage{tikz}
\usepackage{mathtools}
\usepackage{stmaryrd}
\usepackage[boxsize=0.8em, centertableaux]{ytableau}

\title{The Frobenius transform of a symmetric function}
\author{Mitchell Lee}
\address{Harvard University\\
Dept. of mathematics\\
1 Oxford st.\\
Cambridge\\ MA 02138 (USA)}
\email{mitchell@math.harvard.edu}
\keywords{symmetric functions, representation theory of categories, plethysm, Kronecker product}
\subjclass{05E05}

\newcommand{\statement}[1]{%
  #1\enspace\ignorespaces
}

\newcommand*{\ldblbrace}{\{\mskip-5mu\{}
\newcommand*{\rdblbrace}{\}\mskip-5mu\}}
\DeclareRobustCommand{\stirlingi}{\genfrac[]{0pt}{}}
\DeclareRobustCommand{\stirlingii}{\genfrac\{\}{0pt}{}}

\newcommand{\K}{\mathbb{C}}
\newcommand{\Z}{\mathbb{Z}}
\newcommand{\N}{\mathbb{N}}

\newcommand{\Vect}{\text{Vect}}

\newcommand{\Mod}{\mathrm{Mod}}
\newcommand{\Fa}{\mathscr{F}}

\newcommand{\Sur}{\mathchoice{\operatorname{Sur}}{\operatorname{Sur}}{\mathrm{Sur}}{\mathrm{Sur}}}
\newcommand{\Bij}{\mathchoice{\operatorname{Bij}}{\operatorname{Bij}}{\mathrm{Bij}}{\mathrm{Bij}}}
\newcommand{\rmSur}{\mathrm{Sur}}
\newcommand{\rmBij}{\mathrm{Bij}}
\newcommand{\rmFun}{\mathrm{Fun}}

\newcommand{\F}[1]{\mathscr{F}\{#1\}}
\newcommand{\FSur}[1]{\mathscr{F}_{\Sur}\left\{#1\right\}}
\newcommand{\FSurInv}[1]{\mathscr{F}_{\Sur}^{-1}\left\{#1\right\}}

\newcommand{\Sym}{\mathfrak{S}}

\DeclareMathOperator{\ch}{ch}
\DeclareMathOperator{\Hom}{Hom}
\DeclareMathOperator{\Lyndon}{Lyndon}

\def\multichoose#1#2{\ensuremath{\left(\kern-.3em\left(\genfrac{}{}{0pt}{}{#1}{#2}\right)\kern-.3em\right)}}

\theoremstyle{definition}
\newtheorem{defi}{Definition}
\newtheorem{exam}{Example}

\theoremstyle{plain}
\newtheorem{prop}{Proposition}
\newtheorem{theorem}{Theorem}
\newtheorem{lemma}{Lemma}
\newtheorem{coro}{Corollary}

\theoremstyle{remark}
\newtheorem{rema}{Remark}

\begin{document}

\begin{abstract}
    We define an abelian group homomorphism~$\Fa$, which we call the \emph{Frobenius transform}, from the ring of symmetric functions to the ring of the symmetric power series. The matrix entries of~$\Fa$ in the Schur basis are the \emph{restriction coefficients} $r_\lambda^\mu = \dim \operatorname{Hom}_{\Sym_n}(V_\mu, \mathbb{S}^\lambda \K^n)$, which are known to be nonnegative integers but have no known combinatorial interpretation. The Frobenius transform satisfies the identity $\F{fg} = \F{f} \ast \F{g}$, where~$\ast$ is the Kronecker product.

    We prove for all symmetric functions~$f$ that $\F{f} = \FSur{f} \cdot (1 + h_1 + h_2 + \cdots)$, where $\FSur{f}$ is a symmetric function with the same degree and leading term as~$f$. Then, we compute the matrix entries of $\Fa_{\Sur}$ in the complete homogeneous, elementary, and power sum bases and of $\Fa_{\Sur}^{-1}$ in the complete homogeneous and elementary bases, giving combinatorial interpretations of the coefficients where possible. In particular, the matrix entries of $\Fa^{-1}_{\Sur}$ in the elementary basis count words with a constraint on their Lyndon factorization.
    
    As an example application of our main results, we prove that $r_\lambda^\mu = 0$ if $|\lambda \cap \hat{\mu}| < 2|\hat{\mu}| - |\lambda|$, where $\hat{\mu}$ is the partition formed by removing the first part of~$\mu$. We also prove that $r_\lambda^\mu = 0$ if the Young diagram of~$\mu$ contains a square of side length greater than $2^{\lambda_1 - 1}$, and this inequality is tight.
\end{abstract}
\maketitle

\section{Introduction}
Let $n \geq 0$ and let~$\lambda$ be a partition with at most~$n$ parts. There is a corresponding irreducible $GL_n(\K)$-module: the Schur module $\mathbb{S}^\lambda \K^n$. Because the symmetric group $\Sym_n$ embeds in $GL_n(\K)$ by permutation matrices, one may ask: how does the restriction of $\mathbb{S}^\lambda \K^n$ to $\Sym_n$ decompose into irreducible $\Sym_n$-modules?

In other words, let~$\lambda$ and~$\mu$ be partitions and let $n = |\mu|$. What is the value of the \textit{restriction coefficient} \[r_{\lambda}^\mu = \dim \operatorname{Hom}_{\Sym_n}(V_\mu, \mathbb{S}^\lambda \K^n),\] where $V_\mu$ is the Specht module corresponding to the partition~$\mu$? This problem, called the \emph{restriction problem}, has held considerable recent interest \cite{MR4055931, MR4356269, MR4311085, MR4295089, MR4127906}. However, there remains no known combinatorial interpretation for $r_\lambda^\mu$.

Let~$\Lambda$ be the ring of symmetric functions in the variables $x_1, x_2, x_3, \ldots$ and let $\overline{\Lambda}$ be the ring of symmetric power series in $x_1, x_2, x_3, \ldots$. In this paper, we will consider the abelian group homomorphism $\Fa \colon \Lambda \to \overline{\Lambda}$ defined on the basis $\{s_\lambda\}$ of Schur functions by \[\F{s_\lambda} = \sum_{\mu} r_{\lambda}^\mu s_\mu.\] Equivalently, $\F{s_\lambda}$ is the result of applying the Frobenius character map to the representation $\bigoplus_{n} \mathbb{S}^\lambda \K^n$ of $\bigoplus_{n} \K[\Sym_n]$. For this reason, we call~$\Fa$ the \emph{Frobenius transform}. It encodes all information about all the restriction coefficients.

\Cref{section:basics} will cover the basic properties of the Frobenius transform, many of which are implicit in the work of Orellana and Zabrocki. For example, for any symmetric functions $f, g$, we have $\F{fg} = \F{f} \ast \F{g}$, where~$\ast$ is the Kronecker product of symmetric functions. Moreover, for any $f \in \Lambda$, there exists $\FSur{f} \in \Lambda$ with the same degree and leading term as~$f$ such that $\F{f} = \FSur{f} \cdot (1 + h_1 + h_2 + \cdots)$. We refer to the map $\Fa_{\Sur} \colon \Lambda \to \Lambda$ as the \emph{surjective Frobenius transform}. Because the surjective Frobenius transform preserves degree and leading term, it has an inverse, which we denote by $\Fa_{\Sur}^{-1}$.

The \emph{induced trivial character basis} $\{\tilde{h}_\lambda\}_\lambda$ and the \emph{irreducible character basis} $\{\tilde{s}_\lambda\}_\lambda$, which were introduced by Orellana and Zabrocki in 2021 \cite{MR4295089}, can be defined in terms of the inverse surjective Frobenius transform. Namely, $\tilde{h}_\lambda = \FSurInv{h_\lambda}$ and $\tilde{s}_\lambda = \FSurInv{(1 + h_1 + h_2 + \cdots)^{\perp} s_\lambda}$, where $f^\perp\colon \Lambda \to \Lambda$ denotes the operator adjoint under the Hall inner product to multiplication by~$f$.

In \Cref{section:stable}, we will use the Frobenius transform to study \emph{stable restriction coefficients}; that is, the limits \[a_\lambda^\mu = \lim_{n \to \infty} r_{\lambda}^{(n - |\mu|, \mu_1, \ldots, \mu_{\ell(\mu)})},\] which exist for all $\lambda, \mu$ by a classical result of Littlewood \cite{MR1576896}. In 2019, Assaf and Speyer found a formula for $a_\lambda^\mu$ and for the entries $b_\lambda^\mu$ of the inverse matrix $[b_\lambda^\mu] = [a_\lambda^\mu]^{-1}$ \cite{MR4055931}. We will provide an alternative proof of these formulas using the Frobenius transform. In \Cref{theorem:summary}, we will broadly summarize the known relationships between the five kinds of restriction coefficients considered in this paper.

In \Cref{section:expansion}, we will prove the following theorem, which shows how to write $\Fa_{\Sur} \colon \Lambda \to \Lambda$ as a sum of operators of the form $fg^\perp$.
\begin{restatable}{theorem}{expansion}\label{theorem:expansion}
    Let~$f$ be a symmetric function. Then
    \[
        \FSur{f} = \sum_{\lambda} s_\lambda (s_\lambda[h_2 + h_3 + h_4 + \cdots])^\perp f,
    \]
    where the sum is over all partitions~$\lambda$.
\end{restatable}
Since \Cref{theorem:expansion} involves plethysm and there is no known simple formula for the plethysm of Schur functions, it is not well-suited for general computation. We will, however, use it to prove that many restriction coefficients $r_\lambda^\mu$ and stable restriction coefficients $a_\lambda^\mu$ vanish.
\begin{restatable}{theorem}{vanishing}\label{theorem:vanishing}
    Let $\lambda, \mu$ be partitions. If $r_\lambda^\mu > 0$, then $|\lambda \cap \hat \mu| \geq 2|\hat \mu| - |\lambda|$, where $\hat \mu = (\mu_2, \ldots, \mu_{\ell(\mu)})$ is the partition formed by removing the first part of~$\mu$.
\end{restatable}
\begin{restatable}{theorem}{stablevanishing}\label{theorem:stablevanishing}
    Let $\lambda, \mu$ be partitions. If $a_\lambda^\mu > 0$, then $|\lambda \cap \mu| \geq 2|\mu| - |\lambda|$.
\end{restatable}

In \Cref{section:computations}, we will compute $\FSur{f}$ when~$f$ is a complete homogeneous, elementary, or power sum symmetric function. We have $\FSur{e_n} = e_n$ and $\FSur{p_n} = \sum_{d \mid n} p_d$ for $n \geq 1$. More generally, with $\N = \{0, 1, 2, \ldots\}$:

\begin{restatable}{theorem}{computation}\label{theorem:computation}
Let~$\lambda$ be a partition and let $\ell = \ell(\lambda)$ be its length.
\begin{enumerate}[label=(\alph*)]
    \item \label{item:fsurh} Then \[\FSur{h_\lambda} = \sum_{M} \prod_{j \in \N^\ell} h_{M(j)},\] where the sum is over all functions $M\colon \N^\ell \to \N$ such that $M(0, \ldots, 0) = 0$ and $\sum_{j \in \N^\ell} j_i M(j) = \lambda_i$ for $i = 1, \ldots, \ell$.
    \item \label{item:fsure} Then \[\FSur{e_\lambda} = \sum_{M} \prod_{j \in \{0, 1\}^\ell}\begin{cases}
    h_{M(j)} & \mbox{if $j_1 + \cdots + j_\ell$ is even;} \\
    e_{M(j)} & \mbox{if $j_1 + \cdots + j_\ell$ is odd,}
    \end{cases}\] where the sum is over all functions $M \colon \{0, 1\}^\ell \to \N$ such that $M(0, \ldots, 0) = 0$ and $\sum_{j \in \{0, 1\}^\ell} j_i M(j) = \lambda_i$ for $i = 1, \ldots, \ell$.
    \item \label{item:fsurp} Then \[\FSur{p_\lambda} = \sum_{\pi} \prod_{U \in \pi}\left(\sum_{d \mid \gcd\{\lambda_i \colon i \in U\}} d^{|U| - 1} p_d\right),\] where the outer sum is over all partitions~$\pi$ of $\{1, \ldots, \ell\}$ into nonempty sets.
\end{enumerate}
\end{restatable}

A statement equivalent to parts (a) and (b) of this theorem has appeared previously in the work of Orellana and Zabrocki \cite[Equation~(6)]{MR4245137}.

\Cref{theorem:computation} has the following interesting consequence. For any partition~$\mu$, denote by $D(\mu)$ the size of the Durfee square of~$\mu$; that is, $D(\mu)$ is the largest integer~$d$ such that $\mu_{d} \geq d$ \cite[Chapter~8]{MR2122332}.

\begin{restatable}{theorem}{durfeebound}\label{theorem:durfeebound}
    Let~$\mu$ be a partition and let $k \geq 1$ be an integer. The following are equivalent:
    \begin{enumerate}[label=(\Alph*)]
        \item There exists a partition~$\lambda$ such that $\lambda_1 \leq k$ and $r_{\lambda}^\mu > 0$.
        \item $D(\mu) \leq 2^{k - 1}$.
    \end{enumerate}
\end{restatable}

In particular, $r_{\lambda}^\mu = 0$ if $D(\mu) > 2^{\lambda_1 - 1}$.

Finally, in \Cref{section:invcomputations}, we will compute $\FSurInv{e_\lambda}$ and $\FSurInv{h_\lambda}$ (\Cref{theorem:invcomputation}). In particular, we will prove that \[\FSurInv{e_\lambda} = \sum_{\mu} (-1)^{|\lambda| - |\mu|} L_\lambda^\mu e_\mu,\] where $L_\lambda^\mu$ is a nonnegative integer with an explicit combinatorial interpretation involving Lyndon words (\Cref{corollary:fsurinve}).

\section{Preliminaries}\label{section:preliminaries}
Apart from the definition of the restriction coefficients $a_\lambda^\mu$ (\Cref{definition:restriction} below), all the definitions in this section can be found in any standard reference on the theory of symmetric functions \cite{MR1464693} \cite[Chapter~I]{MR3443860} \cite[Chapter~7]{MR4621625}.

Let $\Lambda = \Lambda_\mathbb{Z}$ denote the ring of symmetric functions over $\mathbb{Z}$ in the variables $x_1, x_2, x_3, \ldots$. For $n \geq 0$, let $\Lambda_n$ denote the subgroup of~$\Lambda$ consisting of all symmetric functions that are homogeneous of degree~$n$. Let $\overline{\Lambda}$ denote the ring of symmetric formal power series (that is, formal sums $\sum_{n} f_n$, where each $f_n \in \Lambda_n$). Let $\langle \cdot, \cdot \rangle \colon \Lambda \times \overline{\Lambda} \to \Z$ denote the Hall inner product.

For any partition $\lambda = (\lambda_1, \ldots, \lambda_\ell)$, define the \emph{length} $\ell(\lambda) = \ell$ and the \emph{size} $|\lambda| = \lambda_1 + \cdots + \lambda_\ell$. Let $m_i(\lambda)$ be the number of times the part~$i$ appears in~$\lambda$, and let $z_\lambda = \prod_i i^{m_i(\lambda)} (m_i(\lambda))!$. Let $\lambda^T$ denote the dual (i.e. transpose) of~$\lambda$. Let $m_\lambda, e_\lambda, h_\lambda, p_\lambda, s_\lambda \in \Lambda$ denote the monomial, elementary, homogeneous, power sum, and Schur symmetric functions respectively. Let $V_\lambda$ denote the corresponding \emph{Specht module}, which is an irreducible $\Sym_{|\lambda|}$-module. Let $\chi_\lambda$ denote the character of $V_\lambda$. Let $\mathbb{S}^\lambda$ denote the corresponding \emph{Schur functor}, which is an endofunctor of the category of vector spaces over~$\K$.

For any partitions $\lambda, \mu$, define the \emph{intersection} $\lambda \cap \mu$ by $\ell(\lambda \cap \mu) = \min(\ell(\lambda), \ell(\mu))$ and $(\lambda \cap \mu)_i = \min(\lambda_i, \mu_i)$. That is, it is the partition whose Young diagram is the intersection of the Young diagrams of~$\lambda$ and~$\mu$. Let us say that $\lambda \subseteq \mu$ if $\lambda \cap \mu = \lambda$.

Let $\omega \colon \Lambda \to \Lambda$ be the ring homomorphism given by $\omega(p_k) = (-1)^{k-1} p_k$. Recall that~$\omega$ is an involution and that for all partitions~$\lambda$, we have $\omega(h_\lambda) = e_\lambda$ and $\omega(s_\lambda) = s_{\lambda^T}$.

The \emph{Lyndon symmetric function} $L_n$ is given by \[L_n = \frac{1}{n} \sum_{d \mid n} \mu(d) p_d^{n/d}\] for $n \geq 1$, where $\mu \colon \{1, 2, 3, \ldots\} \to \{-1, 0, 1\}$ is the M\"obius function.

For all $f \in \overline{\Lambda}$, the \emph{skewing operator} $f^\perp \colon \Lambda \to \Lambda$ is the adjoint to multiplication by~$f$ under the Hall inner product: \[\langle g, f^\perp h \rangle = \langle fg, h \rangle.\]

We say that a Schur function $s_\lambda$ \emph{appears} in $f \in \overline{\Lambda}$ if $\langle s_\lambda, f \rangle \neq 0$. We say that~$f$ is \emph{Schur positive} if $\langle s_\lambda, f \rangle \geq 0$ for all partitions~$\lambda$.

Let $C_n$ denote the space of all~$\K$-valued class functions on $\Sym_n$. Let $R_n$ denote the additive group of all virtual characters on $\Sym_n$. In other words, $R_n$ is the subgroup of $C_n$ generated by the irreducible characters $\chi_\lambda$. The~$n$th \emph{Frobenius character map} is the map $\ch_n \colon R_n \to \Lambda_n$ defined by \[\ch_n(\chi) = \frac{1}{n!}\sum_{w \in \Sym_n} \chi(w) p_{c(w)},\] where $c(w)$ denotes the cycle type of~$w$. It is well-known that $\ch_n$ is an isomorphism and that $\ch_n(\chi_\lambda) = s_\lambda$ for all~$\lambda$ with $|\lambda| = n$.

The \emph{Kronecker product} is the unique bilinear operator $\ast \colon \Lambda \times \Lambda \to \Lambda$ satisfying $p_\lambda \ast p_\mu = \delta_{\lambda \mu} z_\lambda p_\lambda$ for all $\lambda, \mu$. It extends to a bilinear operator $\ast \colon \overline{\Lambda} \times \overline{\Lambda} \to \overline{\Lambda}$ by continuity. The Frobenius character map and the Kronecker product are related in the following way: for any $\chi_1, \chi_2 \in R_n$, we have $\ch(\chi_1 \chi_2) = \ch(\chi_1) \ast \ch(\chi_2)$.

For any $f, g \in \overline{\Lambda}$, let $f[g]$ denote the plethysm of~$f$ by~$g$. This is well-defined as long as $f \in \Lambda$ or~$g$ has no constant term.

\begin{prop}[{Plethystic Addition Formula, \cite[Section~3.2]{MR2765321}}]\label{proposition:plethysticaddition}
    Let~$\lambda$ be a partition and let $f, g \in \overline{\Lambda}$. Then
    \[s_{\lambda}[f + g] = \sum_{\mu} s_{\lambda / \mu}[f] s_\mu[g],\] where the sum is over all partitions~$\mu$.
\end{prop}

For a variable $t$, let
\begin{alignat*}{4}
H(t) &= \sum_{n \geq 0} h_n t^n &&= \prod_{i \geq 1} \frac{1}{1 - x_i t} &&= \exp\left(\sum_{k \geq 1} \frac{p_k}{k}t^k\right)&&\in \Lambda \llbracket t \rrbracket \\
E(t) &= \sum_{n \geq 0} e_n t^n &&= \prod_{i \geq 1} (1 + x_i t) &&= \exp\left(\sum_{k \geq 1} \frac{p_k}{k} (-1)^{k-1} t^k\right) &&\in \Lambda \llbracket t \rrbracket.
\end{alignat*}
It is clear that $E(t) = \frac{1}{H(-t)}$. Let $H = H(1) = 1 + h_1 + h_2 + \cdots$ and $H_+ = h_1 + h_2 + h_3 + \cdots = H - 1$.

For any partitions $\lambda, \mu, \nu$, the \emph{Littlewood-Richardson coefficient} $c^\nu_{\lambda\mu}$ is the Hall inner product $\langle s_\nu, s_\lambda s_\mu \rangle$. Define $g_{\lambda\mu\nu} = \langle s_\nu, s_\lambda \ast s_\mu \rangle$.

\begin{defi}\label{definition:restriction}
    Let $\lambda, \mu$ be partitions and let $n = |\mu|$. Then, the Schur module $\mathbb{S}^\lambda \K^n$ can be considered as an $\Sym_n$-module, with $\Sym_n$ acting on $\K^n$ by permutation matrices. The \emph{restriction coefficient} \[r_\lambda^\mu = \dim \operatorname{Hom}_{\Sym_n}(V_\mu, \mathbb{S}^\lambda \K^n)\] is the multiplicity of the Specht module $V_\mu$ in $\mathbb{S}^\lambda \K^n$. 
\end{defi}

\section*{Acknowledgements}
The author thanks Rosa Orellana and Mike Zabrocki for helpful correspondence.

\section{The Frobenius transform: definition and basic properties}
\label{section:basics}
Recall from the introduction that it is a long-standing open problem to find a combinatorial interpretation of $r_\lambda^\mu$. As a potential way to approach this problem, we now define the Frobenius transform, which is the primary object of study in this paper.
\begin{defi}
    The \emph{Frobenius transform} is the abelian group homomorphism $\Fa \colon \Lambda \to \overline{\Lambda}$ defined on the basis $\{s_\lambda\}$ by \[\F{s_\lambda} = \sum_\mu r_{\lambda}^\mu s_\mu,\] where the sum is over all partitions~$\mu$.
\end{defi}
\begin{rema}\label{remark:halladjoint}
By a classical result of Littlewood \cite{MR1576896}, we have \[r_\lambda^\mu = \langle s_\lambda, s_\mu[H] \rangle.\] Hence,~$\Fa$ is adjoint to plethysm by~$H$ under the Hall inner product.
\end{rema}
Here is the reason for calling~$\Fa$ the Frobenius transform. Let $n \geq 0$ and let~$\lambda$ be any partition. Then $\mathbb{S}^\lambda \K^n$, considered as an $\Sym_n$-module, can be expressed as a direct sum of Specht modules:
\[\bigoplus_{|\mu| = n} r_{\lambda}^\mu V_\mu = \mathbb{S}^\lambda \K^n.\]
Taking the character of both sides and applying the Frobenius character map, we obtain \begin{equation}\label{equation:fslambdan1}\sum_{|\mu| = n} r_{\lambda}^\mu s_\mu = \ch_n(\chi_{\mathbb{S}^\lambda \K^n}).\end{equation} In other words, the degree~$n$ part of $\F{s_\lambda}$ is equal to the Frobenius character of $\mathbb{S}^\lambda \K^n$.

\begin{exam}\label{example:fer}
    Let $r \geq 0$. We will compute $\F{e_r}$. First, $e_r = s_{\lambda}$, where $\lambda = (1^r)$. Hence, for any~$n$, the degree~$n$ part of $\F{e_r}$ is the Frobenius character of \[\mathbb{S}^{\lambda} \K^n = \wedge^r \K^n = \operatorname{Ind}_{\Sym_r \times \Sym_{n - r}}^{\Sym_n} (V_{(1^r)} \otimes V_{(n-r)}),\] considered as an $\Sym_n$-module. Thus, it is equal to $e_r h_{n - r}$ \cite[Proposition~7.18.2]{MR4621625}. Taking the sum over all~$n$ yields $\F{e_r} = e_r \cdot H$.
\end{exam}

\subsection{The Frobenius transform and representations of combinatorial categories}\label{subsection:category}
The purpose of this subsection is to provide an alternate perspective on the Frobenius transform. This subsection is not essential to the proofs of our main results, so the reader may skip it. For a more complete introduction to the representation theory of categories, see Wiltshire-Gordon's 2016 PhD thesis \cite{MR3641124}.

In what follows, let $\Vect_\K$ be the category whose objects are vector spaces over~$\K$ and whose morphisms are linear transformations. (The ground field~$\K$ can be replaced by any algebraically closed field of characteristic $0$.)

\begin{defi}
    Let $\mathcal{C}$ be a category. A \emph{$\mathcal{C}$-module} (over the ground field~$\K$) is a functor $M[\bullet] \colon \mathcal{C} \to \Vect_\K$. The \emph{category of $\mathcal{C}$-modules} is the functor category $\Mod^\mathcal{C} = (\Vect_\K)^\mathcal{C}$.
    
    We say that a $\mathcal{C}$-module~$M$ is \emph{finite-dimensional} if $M[U]$ is finite-dimensional for all $U \in \operatorname{Ob}(\mathcal{C})$.

    If~$M$ and~$N$ are $\mathcal{C}$-modules, we may form the \emph{direct sum} $M \oplus N$ by $(M \oplus N)[U] = M[U] \oplus N[U]$.
\end{defi}

Let $\rmBij$ be the category whose objects are finite sets and whose morphisms are bijections, and let~$M$ be a $\rmBij$-module. (Previous authors have referred to $\rmBij$-modules as \emph{linear species} or \emph{tensor species} \cite{MR0927763, MR1132426}.) For $n \geq 0$, denote by $M[n]$ the vector space $M[[n]] = M[\{1, \ldots, n\}]$, which is an $\Sym_n$-module. Then~$M$ is uniquely determined, up to natural isomorphism, by the sequence $M[0], M[1], M[2], \ldots$ of symmetric group modules; namely, 
\begin{equation}\label{equation:bmodule}
M[U] = \K \Bij([n], U) \otimes_{\Sym_n} M[n]
\end{equation} for all finite sets~$U$ with $|U| = n$.
Additionally, if~$M$ is finite-dimensional, then we may form the series
\[\ch(M) = \sum_{n} \ch_n(\chi_{M[n]}) \in \overline{\Lambda},\] which we call the \emph{Frobenius character of~$M$}. It also uniquely determines~$M$ up to natural isomorphism.

Not every symmetric power series can be written in the form $\ch(M)$, where~$M$ is a finite-dimensional $\rmBij$-module. (Every symmetric power series can be written as the Frobenius character of a \emph{virtual $\rmBij$-module}, but we will not define virtual $\rmBij$-modules in this article.) However, it is still helpful to think of~$\ch$ as a partial correspondence between (isomorphism classes of) finite-dimensional $\rmBij$-modules and symmetric power series. Many concepts from the theory of symmetric power series have analogous concepts in the theory of $\rmBij$-modules. For example, in \Cref{definition:speciesproduct} below, we will define the \emph{product} $M \cdot N$ of two $\rmBij$-modules $M, N$. Via the Frobenius character, this is analogous to the product of symmetric power series in the sense that $\ch(M \cdot N) = \ch(M) \ch(N)$ (\Cref{proposition:moduleproduct}).

In what follows, we will find the construction in the theory of $\rmBij$-modules which is analogous to the Frobenius transform. More precisely, let~$M$ be a $\rmBij$-module such that $M_n = 0$ for all but finitely many~$n$. Then $\ch(M)$ is in fact a symmetric function, so its Frobenius transform $\F{\ch(M)}$ is well-defined. We will construct a $\rmBij$-module whose Frobenius character is $\F{\ch(M)}$. Before we do, we need the following definitions.

\begin{defi}[{\cite[Definition~2.6.1]{MR3641124}}]
    Let $F \colon \mathcal{C} \to \mathcal{D}$ be a functor. The \emph{pullback functor} $F^* \colon \Mod^\mathcal{D} \to \Mod^\mathcal{C}$ is given by $F^*M = M \circ F$.
\end{defi}
\begin{defi}[{\cite[Proposition~2.6.2]{MR3641124}}]
    Let $F \colon \mathcal{C} \to \mathcal{D}$ be a functor. The \emph{left Kan extension functor} $F_! \colon \Mod^\mathcal{C} \to \Mod^\mathcal{D}$ is the left adjoint to the pullback $F^*$, if it exists:
    \[\Hom_{\Mod^\mathcal{D}}(F_! M, N) = \Hom_{\Mod^\mathcal{C}}(M, F^*N).\]
\end{defi}

Let $\rmFun$ be the category whose objects are finite sets and whose morphisms are functions. Clearly, $\rmBij$ is a subcategory of $\rmFun$; let $\iota \colon \rmBij \to \rmFun$ be the inclusion functor.

\begin{prop}\label{proposition:leftkanextension}
    The left Kan extension functor $\iota_! \colon \Mod^\rmBij \to \Mod^\rmFun$ exists. Moreover, for every $\rmBij$-module~$M$, the left Kan extension $\iota_! M$ is given on finite sets~$U$ by \[(\iota_! M)[U] = \bigoplus_{n} \K U^n \otimes_{\Sym_n} M[n].\] Here $\K U^n \otimes_{\Sym_n} M[n]$ denotes the quotient of the tensor product $\K U^n \otimes M[n]$ by the relation that $a \otimes w b \sim aw \otimes b$ for all $a \in \K U^n$, $b \in M_n$, and $w \in \Sym_n$.
\end{prop}
\begin{proof}
    This follows from \cite[Proposition~2.6.7]{MR3641124}.
\end{proof}
\begin{rema}
    Joyal refers to $\iota_! M \colon \rmFun \to \Vect_\K$ as the \emph{analytic functor} corresponding to the tensor species~$M$ \cite[Definition~4.2]{MR0927763}.
\end{rema}
We are now ready to show that $\iota^* \iota_!$ is the construction analogous to the Frobenius transform. 
\begin{prop}\label{proposition:kanfrobenius}
    Let~$M$ be a finite-dimensional $\rmBij$-module such that $M[n] = 0$ for all but finitely many~$n$. Then $\iota^*\iota_! M$ is finite-dimensional and $\F{\ch(M)} = \ch(\iota^*\iota_! M)$.
\end{prop}
\begin{proof}
    For any partition~$\lambda$, define the $\rmBij$-module $M_\lambda$ by \[(M_\lambda)[n] = \begin{cases}
        V_\lambda & \mbox{if $n = |\lambda|$;}\\
        0 & \mbox{otherwise,}
    \end{cases}\] extending to all of $\rmBij$ using \eqref{equation:bmodule}.
    
    We have that~$M$ can be written as a direct sum of the $M_\lambda$. Since $\iota^*$ and $\iota_!$ preserve direct sums, it is enough to prove the proposition for $M = M_\lambda$. In this case, $\ch(M) = \ch_n(\chi_\lambda) = s_\lambda$.

    On the other hand, by \Cref{proposition:leftkanextension}, we have for $m \geq 0$ that 
    \begin{align*}
        (\iota^*\iota_! M)[m] &= \bigoplus_{n} \K [m]^n \otimes_{\Sym_n} M[n] \\
        &= \bigoplus_{n} (\K^m)^{\otimes n} \otimes_{\Sym_n} M[n] \\
        &= (\K^m)^{\otimes |\lambda|} \otimes_{\Sym_{|\lambda|}} V_\lambda.
    \end{align*}
    By Schur-Weyl duality, this is isomorphic as a $GL_m$-module to $\mathbb{S}^\lambda \K^m$. Hence, as a $\Sym_m$-module, it decomposes into irreducibles as
    \[(\iota^* \iota_! M)[m] = \bigoplus_{|\mu| = m} r^\mu_\lambda V_\mu.\] It follows that \[\ch(\iota^* \iota_! M) = \sum_{\mu} r^{\mu}_\lambda s_\mu = \F{s_\lambda} = \F{\ch(M)}\] as desired.
\end{proof}
We will occasionally make use of the following definition in later remarks.
\begin{defi}[{\cite[Section~4.1]{MR0927763}}]\label{definition:speciesproduct}
    Let $M, N$ be $\rmBij$-modules. Define the \emph{product} $M \cdot N$ to be the $\rmBij$-module given by \[(M \cdot N)[U] = \bigoplus_{\substack{U_1 \cup U_2 = U \\ U_1 \cap U_2 = \emptyset}} M[U_1] \otimes N[U_2].\]
\end{defi}
\begin{prop}[{\cite[Proposition~2.1]{MR1132426}}]
    \label{proposition:moduleproduct}
    Let~$M$,~$N$ be finite-dimensional $\rmBij$-modules. Then \[\ch(M \cdot N) = \ch(M) \ch(N).\]
\end{prop}
\subsection{The Frobenius transform and evaluation at roots of unity}\label{subsection:rootsofunity}
Let us now describe the expansion of $\F{s_\lambda}$ in the power sum basis. In order to simplify the description, we will do it one degree at a time. In this subsection, we will write $f_n$ to mean the degree~$n$ part of~$f$.

By \eqref{equation:fslambdan1}, we have
\begin{equation}\label{equation:fslambdan2}
(\F{s_\lambda})_n = \ch_n(\chi_{\mathbb{S}^\lambda \K^n}) = \frac{1}{n!} \sum_{w \in \Sym_n} \chi_{\mathbb{S}^\lambda \K^n}(w) p_{c(w)}.
\end{equation}
It is well-known \cite[Section~8.3]{MR1464693} that the character of $\mathbb{S}^\lambda \K^n$ as a $GL_n$-module is the Schur function $s_\lambda$ itself. In other words, if $g \in GL_n(\K)$ has eigenvalues $x_1, \ldots, x_n$, then $\chi_{\mathbb{S}^\lambda \K^n}(g)$ is equal to the evaluation $s_\lambda(x_1, \ldots, x_n)$.

For $w \in \Sym_n$, we have $\chi_{\mathbb{S}^\lambda \K^n}(w) = s_\lambda(x_1, \ldots, x_n)$, where $x_1, \ldots, x_n$ are the eigenvalues of the permutation matrix $P_w$. Let $\mu = (\mu_1, \ldots, \mu_\ell)$ be the cycle type of~$w$. Then the eigenvalues of $P_w$ are the roots of unity
\begin{gather*}
    1, \exp\left(\frac{2 \pi i}{\mu_1}\right), \exp\left(\frac{4 \pi i}{\mu_1}\right), \ldots, \exp\left(\frac{2(\mu_1 - 1) \pi i}{\mu_1}\right), \\
    \vdots \\
    1, \exp\left(\frac{2 \pi i}{\mu_\ell}\right), \exp\left(\frac{4 \pi i}{\mu_\ell}\right), \ldots, \exp\left(\frac{2(\mu_\ell - 1) \pi i}{\mu_\ell}\right).
\end{gather*}
Following Orellana and Zabrocki \cite{MR4275829, MR4295089}, let $\Xi_\mu \in \K^n$ denote this sequence. We have that $\chi_{\mathbb{S}^\lambda \K^n}(w)$ is the result of evaluating $s_\lambda$ at these roots of unity: \[\chi_{\mathbb{S}^\lambda \K^n}(w) =
s_\lambda\left(\Xi_\mu\right).\]
Now, let us group the terms on the right-hand side of \eqref{equation:fslambdan2} according to the cycle type $\mu = c(w)$. The number of permutations $w \in \Sym_n$ with cycle type~$\mu$ is $\frac{n!}{z_\mu}$, so 
\begin{align*}(\F{s_\lambda})_n &= \frac{1}{n!} \sum_{w \in \Sym_n} \chi_{\mathbb{S}^\lambda \K^n}(w) p_{c(w)} \\
&= \frac{1}{n!} \sum_{|\mu| = n} \frac{n!}{z_\mu}s_\lambda(\Xi_\mu)p_\mu \\
&= \sum_{|\mu| = n} s_\lambda(\Xi_\mu) \frac{p_\mu}{z_\mu}.\end{align*}
Taking the sum over all~$n$, we obtain \[\F{s_\lambda} = \sum_{\mu} s_\lambda(\Xi_\mu) \frac{p_\mu}{z_\mu}.\]
Finally, by linearity, we may extend this result to any $f \in \Lambda$. We have proved the following.
\begin{prop}\label{proposition:rootsofunityformula}
    Let $f \in \Lambda$. Then \[\F{f} = \sum_{\mu} f(\Xi_\mu) \frac{p_\mu}{z_\mu}.\]
\end{prop}
In the notation of Orellana and Zabrocki \cite{MR4275829}, this proposition can be written as \[\F{f} = \phi_0(f) + \phi_1(f) + \phi_2(f) + \cdots.\]

\subsection{The Frobenius transform and the Kronecker product}
The Frobenius transform relates the ordinary product of symmetric functions to the Kronecker product in the following way.
\begin{prop}[{cf. \cite[Section~2.3]{MR4275829}}]
\label{proposition:kronecker}
Let $f, g \in \Lambda$. Then $\F{fg} = \F{f} \ast \F{g}$.
\end{prop}
We provide two different proofs of \Cref{proposition:kronecker}: a category-theoretic proof using \Cref{proposition:kanfrobenius} and a direct computational proof using \Cref{proposition:rootsofunityformula}.
\begin{proof}[First proof]
    Because both sides of the desired equation are bilinear, we may assume that there exist $\rmBij$-modules $M, N$ such that $\ch(M) = f$ and $\ch(N) = g$. By \Cref{proposition:moduleproduct}, we have $\ch(M \cdot N) = \ch(M) \ch(N) = fg$, where $M \cdot N$ is the product as defined in \Cref{definition:speciesproduct}.
    
    By \cite[Equation~2.1(ii)]{MR0927763}, we have $\iota^* \iota_! (M \cdot N) = (\iota^* \iota_! M) \otimes_{\K} (\iota^* \iota_! N)$, where $\otimes_\K$ denotes the object-wise tensor product of $\rmBij$-modules. Applying~$\ch$ to both sides, we obtain $\F{fg} = \F{f} \ast \F{g}$, as desired.
\end{proof}
\begin{proof}[Second proof]
By \Cref{proposition:rootsofunityformula}, we have
\begin{align*}
    \F{f} \ast \F{g} &= \left(\sum_{\mu} f(\Xi_\mu) \frac{p_\mu}{z_\mu}\right) \ast \left(\sum_{\mu} g(\Xi_\mu) \frac{p_\mu}{z_\mu}\right) \\
    &= \sum_{\mu} f(\Xi_\mu)g(\Xi_\mu) \frac{p_\mu}{z_\mu} \\
    &= \F{fg}
\end{align*}
as desired.
\end{proof}

\begin{coro}
\label{corollary:productcoefficients}
Let $\lambda, \mu, \nu$ be partitions. Then
\[\sum_{\nu'} r^{\nu}_{\nu'} c^{\nu'}_{\lambda\mu} = \sum_{\lambda', \mu'} r^{\lambda'}_{\lambda} r^{\mu'}_{\mu} g_{\lambda'\mu'\nu}.\]
\end{coro}
\begin{proof}
In \Cref{proposition:kronecker}, take $f = s_\lambda$ and $g = s_\mu$. Then take the Hall inner product of both sides with $s_\nu$.
\end{proof}

\subsection{The surjective Frobenius transform}
\begin{prop}
\label{proposition:fsur}
Let $f \in \Lambda$. Then there exists a symmetric function $\FSur{f} \in \Lambda$ such that $\F{f} = \FSur{f} \cdot H$. Moreover, $\FSur{f}$ has the same degree and leading term as~$f$.
\end{prop}
For example, in \Cref{example:fer} we showed that $\F{e_r} = e_r \cdot H$. Thus, $\FSur{e_r} = e_r$. (Note, however, that in general, $\Fa_{\Sur}$ does not preserve the property of being homogeneous.) Again, we provide two separate proofs of this proposition: a category-theoretic proof and a direct computational proof.
\begin{proof}[First proof]
Because both sides of the desired equation are linear in~$f$, we may assume that there exists a $\rmBij$-module~$M$ such that $\ch(M) = f$.

Let $\rmSur$ be the category whose objects are finite sets and whose morphisms are surjections, and let $\kappa \colon \rmBij \to \rmSur$ be the inclusion functor. Let $\kappa_! \colon \Mod^\rmBij \to \Mod^\rmFun$ be the left Kan extension functor along~$\kappa$, which exists by \cite[Proposition~2.6.7]{MR3641124}. We claim that the choice $\FSur{f} = \ch(\kappa^* \kappa_!M)$ satisfies the desired properties. For this, we will show that there is a natural isomorphism
\begin{equation}\label{eq:iotakappa}
    \iota^* \iota_! M = \kappa^* \kappa_! M \cdot \K E,
\end{equation}
where~$\iota$ and~$\cdot$ are defined as in \Cref{subsection:category} and $\K E$ is the $\rmBij$-module given by $(\K E)[U] = \K$ for all finite sets~$U$.

By \Cref{proposition:leftkanextension}, we have
\begin{equation}\label{eq:iotaiota}
    (\iota^* \iota_! M)[U] = \bigoplus_n \K U^n \otimes_{\Sym_n} M[n].
\end{equation}
Think of $U^n$ as the set of all functions $[n] \to U$. By grouping those functions by their image~$V$, we may write $U^n$ as a disjoint union:
\[U^n = \sum_{V \subseteq U} \Sur([n], V).\]
Substituting into \eqref{eq:iotaiota}, we obtain
\begin{align*}
    (\iota^* \iota_! M)[U] &= \bigoplus_n \K \left(\sum_{V \subseteq U} \Sur([n], V)\right) \otimes_{\Sym_n} M[n] \\
    &= \bigoplus_n \left(\bigoplus_{V \subseteq U} \K\Sur([n], V)\right) \otimes_{\Sym_n} M[n] \\
    &= \bigoplus_{V \subseteq U} \bigoplus_n \K\Sur([n], V) \otimes_{\Sym_n} M[n] \\
    &= \bigoplus_{V \subseteq U} (\kappa^*\kappa_!M)[V]
\end{align*}
where the last step follows from \cite[Proposition~2.6.7]{MR3641124}. Recognizing the latter as $(\kappa^* \kappa_! M \cdot \K E)[U]$, we have shown \eqref{eq:iotakappa}. Applying~$\ch$ to both sides and using \Cref{proposition:kanfrobenius} and \Cref{proposition:moduleproduct} yields $\F{f} = \ch(\kappa^* \kappa_!M) \cdot H$.

It remains to show that $\ch(\kappa^* \kappa_!M)$ has the same degree and leading term as~$f$. For this, we will prove that $(\kappa^* \kappa_!M)[m]$ and $M[m]$ are isomorphic as $\Sym_m$-modules for all $m \geq \deg f$. Consider the natural isomorphism \[(\kappa^*\kappa_!M)[U] = \bigoplus_n \K\Sur([n], U) \otimes_{\Sym_n} M[n]\] with $U = [m]$. In each summand, the factor $\K\Sur([n], [m])$ vanishes when $n < m$ and the factor $M[n]$ vanishes when $n > \deg f$. If $m \geq \deg f$, this means that every term vanishes except possibly the term $n = m$. Hence,
\begin{align*}
    (\kappa^*\kappa_!M)[m] &= \K\Sur([m], [m]) \otimes_{\Sym_m} M[m] \\
    &= \K\Sym_m \otimes_{\Sym_m} M[m] \\
    &= M[m]
\end{align*}
as desired.
\end{proof}
\begin{proof}[Second proof]
We claim that the choice
\[\FSur{f} = \sum_{\mu} \langle f, s_\mu[H_+]\rangle s_\mu\] satisfies the desired properties. In other words, we may take $\Fa_{\Sur}$ to be the adjoint to plethysm by $H_+$ under the Hall inner product.

By \Cref{remark:halladjoint}, we have
\begin{align*}
    \F{f} &= \sum_{\lambda} \langle f, s_\lambda[H]\rangle s_\lambda \\
    &= \sum_{\lambda} \langle f, s_\lambda[H_+ + 1]\rangle s_\lambda.
\end{align*}
By the plethystic addition formula (\Cref{proposition:plethysticaddition}), we have
\begin{align*}
    \F{f} &= \sum_{\lambda, \mu} \langle f, s_\mu[H_+] s_{\lambda / \mu}[1]\rangle s_\lambda.
\end{align*}
We have \[s_{\lambda / \mu}[1] = s_{\lambda / \mu}(1, 0, 0, \ldots) = \begin{cases} 1 & \mbox{if $\lambda / \mu$ is a horizontal strip;} \\ 0 & \mbox{otherwise,}\end{cases}\] so 
\begin{align*}
    \F{f} &= \sum_{\substack{\lambda, \mu \\ \text{$\lambda/\mu$ h. strip}}} \langle f, s_\mu[H_+]\rangle s_\lambda \\
    &= \sum_{\mu}\langle f, s_\mu[H_+] \rangle \left(\sum_{\substack{\lambda \\ \text{$\lambda/\mu$ h. strip}}}s_\lambda\right) \\
    &= \sum_{\mu}\langle f, s_\mu[H_+] \rangle s_\mu \cdot H,
\end{align*}
where the last equality follows from the Pieri rule.

Now, the only thing left to show is that~$f$ has the same degree and leading term as \[\sum_{\mu} \langle f, s_\mu[H_+]\rangle s_\mu.\] This follows directly from the observation that if $|\mu| \geq \deg f$, then $\langle f, s_\mu[H_+]\rangle = \langle f, s_\mu\rangle$.
\end{proof}
We refer to $\Fa_{\Sur} \colon \Lambda \to \Lambda$ as the \emph{surjective Frobenius transform}. Clearly, it is invertible:
\begin{coro}
    There exists a two-sided inverse $\Fa_{\Sur}^{-1} \colon \Lambda \to \Lambda$ of $\Fa_{\Sur}$.
\end{coro}
\begin{proof}
    Define $\mathcal{M} \colon \Lambda \to \Lambda$ by $\mathcal{M}\{f\} = f - \FSur{f}$. By \Cref{proposition:fsur}, we have \[\deg(\mathcal{M}\{f\}) < \deg f\] for any $f \in \Lambda \setminus \{0\}$. Hence, $\mathcal{M}^k\{f\} = 0$ for any $k > \deg f$.

    Define $\Fa_{\Sur}^{-1} \colon \Lambda \to \Lambda$ by \[\FSurInv{f} = f + \mathcal{M}\{f\} + \mathcal{M}^2\{f\} + \cdots.\] This is well-defined by the above, and it is easy to check that it is the two-sided inverse of $\Fa_{\Sur}$.
\end{proof}

Like the ordinary Frobenius transform, the surjective Frobenius transform can be described in terms of its matrix entries in the Schur basis.
\begin{defi}
Let $\lambda, \mu$ be partitions. Define the \emph{surjective restriction coefficient} \[t_\lambda^\mu = \langle \FSur{s_\lambda}, s_\mu\rangle\] and define the \emph{inverse surjective restriction coefficient} \[u_\lambda^\mu = \langle \FSurInv{s_\lambda}, s_\mu\rangle.\]
\end{defi}
By the above, we have $t_\lambda^\mu = u_\lambda^\mu = \delta_{\lambda \mu}$ for $|\mu| \geq |\lambda|$.
\section{Stable restriction}
\label{section:stable}
Define the \emph{stable restriction coefficients} $a_\lambda^\mu$ as follows. For any partition $\mu = (\mu_1, \ldots, \mu_\ell)$ and any $n \geq \mu_1 + |\mu|$, define $\mu^{(n)} = (n - |\mu|, \mu_1, \ldots, \mu_\ell)$. Then, for any partitions $\lambda, \mu$, the stable restriction coefficient $a_\lambda^\mu$ is defined by the following limit, which exists by a classical result of Littlewood \cite{MR1576896}: \[a_\lambda^\mu = \lim_{n \to \infty} r_{\lambda}^{\mu^{(n)}}.\] Littlewood also showed that $a_{\lambda}^\mu = \delta_{\lambda \mu}$ if $|\lambda| \leq |\mu|$, so the infinite matrix $[a_\lambda^\mu]$, with rows and columns indexed by partitions in increasing order of size, is upper unitriangular. In 2019, Assaf and Speyer \cite{MR4055931} found the following formula for the entries $b_\lambda^\mu$ of the inverse matrix $[b_\lambda^\mu] = [a_\lambda^\mu]^{-1}$:
\begin{restatable}[{\cite[Theorem~2]{MR4055931}}]{theorem}{assafspeyer}
\label{theorem:assafspeyer}
Let $\lambda, \mu$ be partitions. Then \[b_\lambda^\mu = (-1)^{|\lambda| - |\mu|} \langle s_{\lambda^T}, s_{\mu^T}[L_1 + L_2 + L_3 + \cdots] \cdot H \rangle.\] In particular, $(-1)^{|\lambda| - |\mu|}b_\lambda^\mu$ is a nonnegative integer.
\end{restatable}
In this section, we will provide an alternative proof of \Cref{theorem:assafspeyer}. In fact, we will prove similar plethystic formulas for all five kinds of restriction coefficients defined so far. These formulas have all been collected into \Cref{theorem:summary} below.

\begin{theorem}\label{theorem:summary}
    Let $\lambda, \mu$ be partitions with $|\lambda| = m$ and $|\mu| = n$.
    \begin{enumerate}[label=(\alph*)]
        \item The restriction coefficient $r_\lambda^\mu$ is given by 
        \begin{align*}
            r_{\lambda}^\mu
            &= \langle \F{s_\lambda}, s_\mu\rangle \\
            &= \langle s_\lambda, s_\mu[H] \rangle.
        \end{align*}
        \item The surjective restriction coefficient $t_\lambda^\mu$ is given by 
        \begin{align*}
            t_{\lambda}^\mu &= \langle \FSur{s_\lambda}, s_\mu\rangle\\
            &= \langle s_\lambda, s_\mu[H_+] \rangle.
        \end{align*}
        \item The inverse surjective restriction coefficient $u_\lambda^\mu$ is given by
        \begin{align*}
            u_\lambda^\mu &= \langle \FSurInv{s_\lambda}, s_\mu \rangle \\
            &= \langle s_\lambda, s_\mu[\omega(L_1) - \omega(L_2) + \omega(L_3) - \cdots] \rangle \\
            &= (-1)^{m - n}\langle s_{\lambda^T}, s_{\mu^T}[L_1 + L_2 + L_3 + \cdots] \rangle.
        \end{align*}
        \item The stable restriction coefficient $a_\lambda^\mu$ is given by
        \begin{align*}
            a_\lambda^\mu &= \langle H^\perp \FSur{s_\lambda}, s_\mu\rangle \\
            &= \langle \FSur{s_\lambda}, s_\mu \cdot H\rangle \\
            &= \langle s_\lambda, (s_\mu \cdot H)[H_+]\rangle.
        \end{align*}
        \item The inverse stable restriction coefficient $b_\lambda^\mu$ is given by
        \begin{align*}
            b_\lambda^\mu &= \langle \FSurInv{(H^\perp)^{-1} s_\lambda}, s_\mu \rangle \\
            &= \langle s_\lambda, s_\mu[\omega(L_1) - \omega(L_2) + \omega(L_3) - \cdots] \cdot (1 - e_1 + e_2 - e_3 + \cdots)\rangle \\
            &= (-1)^{m - n} \langle s_{\lambda^T}, s_{\mu^T}[L_1 + L_2 + L_3 + \cdots] \cdot H \rangle.
        \end{align*}
    \end{enumerate}
\end{theorem}
Before proving \Cref{theorem:summary}, we restate a basic result of plethystic calculus.
\begin{lemma}[{Negation Rule, \cite[Theorem~6]{MR2765321}}]\label{lemma:negation}
    Let $f \in \Lambda$ and $g \in \overline{\Lambda}$. If~$f$ is homogeneous, then \[f[-g] = (-1)^{\deg f} (\omega(f))[g].\]
\end{lemma}
\begin{proof}[Proof of \Cref{theorem:summary}]
    \begin{enumerate}[label=(\alph*)]
        \item The first equality is true by definition. The second follows from \Cref{remark:halladjoint}.
        \item The first equality is true by definition. The second is demonstrated in the second proof of \Cref{proposition:fsur}.
        \item The first equality is true by definition.
        
        For the second, Cadogan showed in 1971 that $\omega(L_1) - \omega(L_2) + \omega(L_3) - \cdots$ is the plethystic inverse of $H_+$ \cite{MR0284377}. In part (b), we showed that $\Fa_{\Sur}$ is adjoint to plethysm by $H_+$. Hence, $\Fa_{\Sur}^{-1}$ is adjoint to plethysm by $\omega(L_1) - \omega(L_2) + \omega(L_3) - \cdots$, as desired.

        For the third, by \Cref{lemma:negation} and the associativity of plethysm, we have
        \begin{align*}
            & \langle s_\lambda, s_\mu[\omega(L_1) - \omega(L_2) + \omega(L_3) - \cdots] \rangle \\
            =& \langle s_\lambda, s_\mu[-(L_1 + L_2 + L_3 + \cdots)[-p_1]] \rangle \\
            =& \langle s_\lambda, (s_\mu[-(L_1 + L_2 + L_3 + \cdots)]) [-p_1] \rangle \\
            =& \langle s_\lambda[-p_1], s_\mu[-(L_1 + L_2 + L_3 + \cdots)] \rangle \\
            =& \langle (-1)^m \omega(s_\lambda)[p_1], (-1)^n \omega(s_\mu)[L_1 + L_2 + L_3 + \cdots] \rangle \\
            =& (-1)^{m - n} \langle s_{\lambda^T}, s_{\mu^T}[L_1 + L_2 + L_3 + \cdots] \rangle 
        \end{align*}
        as desired.
        \item By part (a), we have
        \begin{equation}\label{equation:amulambda}
            a^\mu_{\lambda}
            = \lim_{n \to \infty} \langle \F{s_\lambda}, s_{\mu^{(n)}} \rangle 
            = \lim_{n \to \infty} \langle \FSur{s_\lambda} \cdot H, s_{\mu^{(n)}} \rangle.
        \end{equation}
        Now, we claim that for any $f \in \Lambda$, we have
        \begin{equation}\label{equation:limit}
            \lim_{n \to \infty} \langle f \cdot H, s_{\mu^{(n)}} \rangle = \langle f, s_\mu \cdot H \rangle.
        \end{equation}
        By linearity, it is enough to show \eqref{equation:limit} when $f = s_\nu$ for some partition~$\nu$. In that case, by the Pieri rule,
        \[
            \lim_{n \to \infty} \langle f \cdot H, s_{\mu^{(n)}} \rangle = \lim_{n \to \infty} \begin{cases} 1 & \mbox{if $\mu^{(n)}/\nu$ is a horizontal strip;} \\ 0 & \mbox{otherwise.}\end{cases}
        \]
        For~$n$ sufficiently large, it is easy to see that $\mu^{(n)}/\nu$ is a horizontal strip if and only if $\nu / \mu$ is a horizontal strip. So the above limit is equal to
        \[
            \begin{cases} 1 & \mbox{if $\nu/\mu$ is a horizontal strip;} \\ 0 & \mbox{otherwise,}\end{cases}
        \]
        which is just $\langle f, s_\mu \cdot H \rangle$. Hence, \eqref{equation:limit} indeed holds. Substituting \eqref{equation:limit} into \eqref{equation:amulambda} with $f = \FSur{s_\lambda}$, we obtain \[a^\mu_\lambda = \langle \FSur{s_\lambda}, s_\mu \cdot H \rangle.\] The result now follows from the fact that $\Fa_{\Sur}$ is adjoint to plethysm by $H_+$.
        \item Define the \emph{stable Frobenius transform} $\mathcal{A} \colon \Lambda \to \Lambda$ by $\mathcal{A} \{f\} = H^\perp \FSur{f}$. Part (d) implies that the matrix of $\mathcal{A}$ in the Schur basis is $[a_\lambda^\mu]$. Therefore, the matrix of $\mathcal{A}^{-1}$ in the Schur basis is $[b_\lambda^\mu]$. This proves the first equality.
        
        The second equality then follows from the fact that $\Fa_{\Sur}^{-1}$ is adjoint to plethysm by $\omega(L_1) - \omega(L_2) + \omega(L_3) - \cdots$.
        
        For the third equality, we again use \Cref{lemma:negation}, together with the fact that plethysm by $-p_1$ is an isometry and a ring automorphism:
        \begin{align*}
            &\langle s_\lambda, s_\mu[\omega(L_1) - \omega(L_2) + \omega(L_3) - \cdots] \cdot (1 - e_1 + e_2 - e_3 + \cdots)\rangle \\
            =& \langle s_\lambda, s_\mu[-(L_1 + L_2 + L_3 + \cdots)][-p_1] \cdot (1 - e_1 + e_2 - e_3 + \cdots)\rangle \\
            =& \langle s_\lambda[-p_1], s_\mu[-(L_1 + L_2 + L_3 + \cdots)] \cdot (1 - e_1 + e_2 - e_3 + \cdots)[-p_1]\rangle \\
            =& \langle (-1)^m \omega(s_\lambda)[p_1], (-1)^n \omega(s_\mu)[L_1 + L_2 + L_3 + \cdots] \cdot H\rangle \\
            =& (-1)^{m - n} \langle s_{\lambda^T}, s_{\mu^T}[L_1 + L_2 + L_3 + \cdots] \cdot H\rangle,
        \end{align*}
        as desired. \qedhere
    \end{enumerate}
\end{proof}
\section{An expansion of the Frobenius transform}\label{section:expansion}
In 1999, Zabrocki showed that every abelian group homomorphism from~$\Lambda$ to itself can be written as a sum of operators of the form $f g^\perp$ where $f, g \in \Lambda$ \cite[Corollary~4.11]{MR1817706}. In this section, we will prove the following theorem, which shows how to write $\Fa_{\Sur}$ in this way.
\expansion*
\begin{rema}
    The symmetric power series $s_\lambda[h_2 + h_3 + h_4 + \cdots] \in \overline{\Lambda}$ appearing in \Cref{theorem:expansion} contains only terms of degree at least $2|\lambda|$. Hence, the degree of the summand $s_\lambda (s_\lambda[h_2 + h_3 + h_4 + \cdots])^\perp f$ is at most $\deg(f) - |\lambda|$. In particular, it vanishes if $|\lambda| > \deg(f)$, so the sum in \Cref{theorem:expansion} is finite.
\end{rema}
\begin{proof}[Proof of \Cref{theorem:expansion}]
    Let~$\mu$ be an arbitrary partition. It suffices to show that
    \begin{equation}\label{equation:dualsmu}
        \langle s_\mu, \FSur{f} \rangle = \left \langle s_\mu, \sum_{\lambda} s_\lambda (s_\lambda[h_2 + h_3 + h_4 + \cdots])^\perp f \right \rangle.
    \end{equation}
    By \Cref{theorem:summary}(b) and \Cref{proposition:plethysticaddition}, we have
    \begin{align*}
        \langle s_\mu, \FSur{f} \rangle &= \langle s_\mu[H_+], f \rangle \\
        &= \left\langle \sum_{\lambda} s_{\mu/\lambda} s_{\lambda}[h_2 + h_3 + h_4 + \cdots], f \right\rangle \\
        &= \left\langle\sum_{\lambda} s_{\lambda}[h_2 + h_3 + h_4 + \cdots] s_\lambda^\perp s_\mu, f \right\rangle \\
        &= \left \langle s_\mu, \sum_{\lambda} s_\lambda (s_\lambda[h_2 + h_3 + h_4 + \cdots])^\perp f \right \rangle,
    \end{align*}
    completing the proof of \eqref{equation:dualsmu} and of the theorem.
\end{proof}
Now, we will use \Cref{theorem:summary} and \Cref{theorem:expansion} to study the vanishing of the surjective restriction coefficients $t_\lambda^\mu$, the restriction coefficients $r_\lambda^\mu$, and the restriction coefficients $a_\lambda^\mu$.
\begin{theorem}\label{theorem:surjectivevanishing}
    Let $\lambda, \mu$ be partitions. If $t_\lambda^\mu > 0$, then $|\lambda \cap \mu| \geq 2 |\mu| - |\lambda|$.
\end{theorem}
\begin{proof}
    By the definition of the surjective restriction coefficients, the Schur function $s_\mu$ appears in $\FSur{s_\lambda}$. By \Cref{theorem:expansion}, there exists a partition~$\nu$ such that $s_\mu$ appears in 
    \[
        s_\nu (s_{\nu}[h_2 + h_3 + h_4 + \cdots])^\perp s_\lambda.
    \]
    Hence, there exists a partition~$\rho$ such that $s_\rho$ appears in $(s_{\nu}[h_2 + h_3 + h_4 + \cdots])^\perp s_\lambda$ and $s_\mu$ appears in $s_\nu s_\rho$.

    Since $s_\rho$ appears in $(s_{\nu}[h_2 + h_3 + h_4 + \cdots])^\perp s_\lambda$, we have that $\rho \subseteq \lambda$ and 
    \begin{equation}\label{equation:rhoinequality}
        |\rho| \leq |\lambda| - 2|\nu|.
    \end{equation}
    Since $s_\mu$ appears in $s_\nu s_\rho$, we have that $\rho \subseteq \mu$ and
    \begin{equation}\label{equation:rhoequation}
        |\mu| = |\nu| + |\rho|.
    \end{equation}
    Combining \eqref{equation:rhoinequality} and \eqref{equation:rhoequation}, we obtain $|\rho| \geq 2|\mu| - |\lambda|$. Now, we have $\rho \subseteq \lambda \cap \mu$, so \[|\lambda \cap \mu| \geq |\rho| \geq 2|\mu| - |\lambda|,\] as desired.
\end{proof}
\vanishing*
\begin{proof}
    By the Pieri rule and the definition of $\Fa_{\Sur}$, we have \[r_\lambda^\mu = \sum_{\substack{\nu \\ \text{$\mu/\nu$ is a horizontal strip}}} t_\lambda^\nu.\] Hence, there exists a partition~$\nu$ such that $\mu/\nu$ is a horizontal strip and $t_\lambda^\nu > 0$. By \Cref{theorem:surjectivevanishing}, we have \[|\lambda \cap \nu| \geq 2 |\nu| - |\lambda|.\] Now, $\hat \mu \subseteq \nu$, so 
    \begin{align*}
        |\lambda \cap \hat \mu| &\geq |\lambda \cap \nu| - (|\nu| - |\hat \mu|) \\
        &\geq (2 |\nu| - |\lambda|) - (|\nu| - |\hat \mu|) \\
        &= (2|\hat \mu| - |\lambda|) + (|\nu| - |\hat \mu|) \\
        &\geq 2|\hat \mu| - |\lambda|,
    \end{align*}
    as desired.
\end{proof}
\stablevanishing*
\begin{proof}
    By the Pieri rule and \Cref{theorem:summary}(d), we have
    \[
        a_\lambda^\mu = \sum_{\substack{\nu \\ \text{$\nu/\mu$ is a horizontal strip}}} t_\lambda^\nu.
    \]
    Hence, there exists a partition~$\nu$ such that $\nu/\mu$ is a horizontal strip and $t_\nu^\mu > 0$. By \Cref{theorem:surjectivevanishing}, we have \[|\lambda \cap \nu| \geq 2 |\nu| - |\lambda|.\] Now, $\mu \subseteq \nu$, so
    \begin{align*}
        |\lambda \cap \mu| &\geq |\lambda \cap \nu| - (|\nu| - |\mu|) \\
        &\geq (2 |\nu| - |\lambda|) - (|\nu| - |\mu|)\\
        &= (2|\mu| - |\lambda|) + (|\nu| - |\mu|) \\
        &= 2|\mu| - |\lambda|, 
    \end{align*}
    as desired.
\end{proof}
\section{Computations of the surjective Frobenius transform}
\label{section:computations}
In this section, we will compute $\FSur{f}$ for various symmetric functions~$f$. Recall from the introduction:

\computation*

\begin{exam}\label{example:h}
    Let us use \Cref{theorem:computation}(a) to compute $\FSur{h_{2, 2}}$. First, we list all the functions $M \colon \N^2 \to \N$ such that $M(0, 0) = 0$ and $\sum_{j \in \N^2} j M(j) = (2, 2)$. There are nine such functions $M_1, \ldots, M_9$. Here are all of their nonzero values.\footnote{For readers who are familiar with the language of multisets and multiset partitions \cite{MR4295089}, it can be helpful to remember that such functions~$M$ are in bijection with multiset partitions of $\ldblbrace 1, 1, 2, 2\rdblbrace$. The multiset partition corresponding to the function~$M$ contains $M(j)$ copies of $\ldblbrace 1^{j_1}, 2^{j_2}\rdblbrace$ for all $j \in \N^2$. For example, the function $M_6$ corresponds to the multiset partition $\ldblbrace \ldblbrace 1, 1 \rdblbrace, \ldblbrace 2 \rdblbrace, \ldblbrace 2 \rdblbrace \rdblbrace \Vdash \ldblbrace 1, 1, 2, 2\rdblbrace$.}
    \begin{equation*}
        \begin{gathered}
            M_1(2, 2) = 1 \\
            M_2(1, 1) = 2 \\
            M_3(2, 1) = 1 \qquad M_3(0, 1) = 1 \\
            M_4(1, 2) = 1 \qquad M_4(1, 0) = 1 \\
            M_5(2, 0) = 1 \qquad M_5(0, 2) = 1 \\
            M_6(2, 0) = 1 \qquad M_6(0, 1) = 2 \\
            M_7(0, 2) = 1 \qquad M_7(1, 0) = 2 \\
            M_8(1, 1) = 1 \qquad M_8(1, 0) = 1 \qquad M_8(0, 1) = 1 \\
            M_9(1, 0) = 2 \qquad M_9(0, 1) = 2
        \end{gathered}
    \end{equation*}
    Thus
    \begin{align*}
        \FSur{h_{2, 2}} &= \underbrace{h_1}_{M_1} + \underbrace{h_2}_{M_2} + \underbrace{h_1^2}_{M_3} + \underbrace{h_1^2}_{M_4} + \underbrace{h_1^2}_{M_5} + \underbrace{h_1 h_2}_{M_6} + \underbrace{h_1 h_2}_{M_7} + \underbrace{h_1^3}_{M_8} + \underbrace{h_2^2}_{M_9} \\
        &= h_1 + h_2 + 3h_{1, 1} + 2h_{2, 1} + h_{1, 1, 1} + h_{2, 2}.
    \end{align*}
\end{exam}

\begin{exam}
    Let us use \Cref{theorem:computation}(b) to compute $\FSur{e_{5, 3}}$. First, we list all the functions $M \colon \{0, 1\}^2 \to \N$ such that $M(0, 0) = 0$ and $\sum_{j \in \{0, 1\}^2} j M(j) = (5, 3)$. There are four such functions $M_1, M_2, M_3, M_4$. Here are all of their nonzero values.
    \begin{equation*}
        \begin{gathered}
            M_1(1, 0) = 5 \qquad M_1(0, 1) = 3 \\
            M_2(1, 1) = 1 \qquad M_2(1, 0) = 4 \qquad M_2(0, 1) = 2 \\
            M_3(1, 1) = 2 \qquad M_3(1, 0) = 3 \qquad M_3(0, 1) = 1 \\
            M_4(1, 1) = 3 \qquad M_4(1, 0) = 2 \\
        \end{gathered}
    \end{equation*}
    Thus
    \begin{align*}
        \FSur{e_{5, 3}} &= \underbrace{e_5 e_3}_{M_1} + \underbrace{h_1 e_4 e_2}_{M_2} + \underbrace{h_2 e_3 e_1}_{M_3} + \underbrace{h_3 e_2}_{M_4}.
    \end{align*}
\end{exam}

\begin{exam}
    Let us use \Cref{theorem:computation}(c) to compute $\FSur{p_{15, 10, 6}}$. Take $\lambda = (15, 10, 6)$ and $\ell = 3$. There are five partitions of $[\ell]$ into nonempty sets: $\{\{1, 2, 3\}\}$, $\{\{1, 2\}, \{3\}\}$, $\{\{1, 3\}, \{2\}\}$, $\{\{2, 3\}, \{1\}\}$, and $\{\{1\}, \{2\}, \{3\}\}$. Thus,
    \begin{align*}
        \FSur{p_\lambda} =& \left(\sum_{d \mid \gcd(\lambda_1, \lambda_2, \lambda_3)} d^2 p_d\right) \\
        &+ \left(\sum_{d \mid \gcd(\lambda_1, \lambda_2)} d p_d\right) \left(\sum_{d \mid \lambda_3} p_d\right) \\
        &+ \left(\sum_{d \mid \gcd(\lambda_1, \lambda_3)} d p_d\right) \left(\sum_{d \mid \lambda_2} p_d\right) \\
        &+ \left(\sum_{d \mid \gcd(\lambda_2, \lambda_3)} d p_d\right) \left(\sum_{d \mid \lambda_1} p_d\right) \\
        &+ \left(\sum_{d \mid \lambda_1} p_d\right) \left(\sum_{d \mid \lambda_2} p_d\right) \left(\sum_{d \mid \lambda_3} p_d\right) \\
        =&\,\, p_1 + (p_1 + 5p_5)(p_1 + p_2 + p_3 + p_6) \\
        &+ (p_1 + 3p_3)(p_1 + p_2 + p_5 + p_{10}) \\
        &+ (p_1 + 2p_2)(p_1 + p_3 + p_5 + p_{15}) \\
        &+ (p_1 + p_3 + p_5 + p_{15})(p_1 + p_2 + p_5 + p_{10})(p_1 + p_2 + p_3 + p_6).
    \end{align*}
\end{exam}

\begin{proof}[Proof of \Cref{theorem:computation}]
    \statement{(a)} Let $t_1, \ldots, t_\ell$ be variables and let \[f = H(t_1) \cdots H(t_\ell) \in \Lambda\llbracket t_1, \ldots, t_\ell \rrbracket.\] We will use \Cref{proposition:rootsofunityformula} to compute \[\F{f} \in \overline{\Lambda}\llbracket t_1, \ldots, t_\ell \rrbracket.\]
    For any partition~$\mu$, we have 
    \begin{align*}
        f(\Xi_{\mu}) &= \prod_{i = 1}^\ell H(t_i)(\Xi_\mu) \\
        &= \prod_{i = 1}^\ell \prod_{j = 1}^{\ell(\mu)} \prod_{k = 0}^{\mu_j - 1} \frac{1}{1 - t_i \exp(2 \pi i k / \mu_j)} \\
        &= \prod_{i = 1}^\ell \prod_{j = 1}^{\ell(\mu)} \frac{1}{1 - t_i^{\mu_j}}.
    \end{align*}
    Hence, by \Cref{proposition:rootsofunityformula}, we have
    \[
        \F{f} = \sum_{\mu} \left(\prod_{i = 1}^\ell \prod_{j = 1}^{\ell(\mu)} \frac{1}{1 - t_i^{\mu_j}}\right) \frac{p_\mu}{z_\mu}. 
    \]
    We may recognize the right-hand side as a product of exponentials, and then evaluate the product as follows.
    \begin{align*}
        \F{f} &= \prod_{k} \exp\left(\frac{p_k}{k} \prod_{i=1}^\ell \frac{1}{1 - t_i^k}\right) \\
        &= \prod_{k} \exp\left(\frac{p_k}{k} \sum_{j \in \N^\ell} (t_1^{j_1}\cdots t_\ell^{j_\ell})^k \right) \\
        &= \prod_{j \in \N^\ell} \exp\left(\sum_k \frac{p_k}{k} (t_1^{j_1}\cdots t_\ell^{j_\ell})^k \right) \\
        &= \prod_{j \in \N^\ell} H(t_1^{j_1}\cdots t_\ell^{j_\ell}).
    \end{align*}
    Hence, \[\FSur{f} = \frac{\F{f}}{H} = \prod_{j \in \N^\ell \setminus \{(0, \ldots, 0)\}} H(t_1^{j_1}\cdots t_\ell^{j_\ell}).\]
    Finally, the result follows from taking the coefficient of $t_1^{\lambda_1} \cdots t_\ell^{\lambda_\ell}$ on both sides.
    
    \statement{(b)} Let $t_1, \ldots, t_\ell$ be variables and let \[f = E(t_1) \cdots E(t_\ell) \in \Lambda\llbracket t_1, \ldots, t_\ell \rrbracket.\] We will use \Cref{proposition:rootsofunityformula} to compute \[\F{f} \in \overline{\Lambda}\llbracket t_1, \ldots, t_\ell \rrbracket.\]
    For any partition~$\mu$, we have 
    \begin{align*}
        f(\Xi_{\mu}) &= \prod_{i = 1}^\ell E(t_i)(\Xi_\mu) \\
        &= \prod_{i = 1}^\ell \prod_{j = 1}^{\ell(\mu)} \prod_{k = 0}^{\mu_j - 1} (1 + t_i \exp(2 \pi i k / \mu_j)) \\
        &= \prod_{i = 1}^\ell \prod_{j = 1}^{\ell(\mu)} (1 - (-t_i)^{\mu_j}).
    \end{align*}
    Hence, by \Cref{proposition:rootsofunityformula}, we have
    \[
        \F{f} = \sum_{\mu} \left(\prod_{i = 1}^\ell \prod_{j = 1}^{\ell(\mu)} (1 - (-t_i))^{\mu_j}\right) \frac{p_\mu}{z_\mu}. 
    \]
    We may recognize the right-hand side as a product of exponentials, and then evaluate the product as follows.
    \begin{align*}
        \F{f} &= \prod_{k} \exp\left(\frac{p_k}{k} \prod_{i=1}^\ell (1 - (-t_i)^k)\right) \\
        &= \prod_{k} \exp\left(\frac{p_k}{k} \sum_{j \in \{0, 1\}^\ell} (-1)^{j_1 + \cdots + j_\ell} ((-t_1)^{j_1}\cdots (-t_\ell)^{j_\ell})^k \right) \\
        &= \prod_{j \in \{0, 1\}^\ell} \exp\left(\sum_k \frac{p_k}{k} ((-1)^{j_1 + \cdots + j_\ell})^{k-1} (t_1^{j_1}\cdots t_\ell^{j_\ell})^k \right) \\
        &= \prod_{j \in \{0, 1\}^\ell} \begin{cases}
            H(t_1^{j_1}\cdots t_\ell^{j_\ell}) & \mbox{if $j_1 + \cdots + j_\ell$ is even;} \\
            E(t_1^{j_1}\cdots t_\ell^{j_\ell}) & \mbox{if $j_1 + \cdots + j_\ell$ is odd.}
        \end{cases}
    \end{align*}
    Hence, \[\FSur{f} = \frac{\F{f}}{H} = \prod_{j \in \{0, 1\}^\ell \setminus \{(0, \ldots, 0)\}} \begin{cases}
        H(t_1^{j_1}\cdots t_\ell^{j_\ell}) & \mbox{if $j_1 + \cdots + j_\ell$ is even;} \\
        E(t_1^{j_1}\cdots t_\ell^{j_\ell}) & \mbox{if $j_1 + \cdots + j_\ell$ is odd.}
    \end{cases}\]
    Finally, the result follows from taking the coefficient of $t_1^{\lambda_1} \cdots t_\ell^{\lambda_\ell}$ on both sides.
    
    \statement{(c)} By \Cref{proposition:rootsofunityformula}, we have
    \begin{equation}\label{equation:rootsofunitypowersum}
        \F{p_\lambda} = \sum_{\mu} p_\lambda(\Xi_\mu) \frac{p_\mu}{z_\mu}.
    \end{equation}
    For all~$k$, we have \[p_k(\Xi_\mu) = \sum_{d \mid k} d m_d(\mu),\] so \eqref{equation:rootsofunitypowersum} becomes
    \begin{equation}\label{equation:powersumproduct}
        \F{p_\lambda} = \sum_{\mu} \prod_{i = 1}^\ell \left(\sum_{d \mid \lambda_i} d m_d(\mu) \right)\frac{p_\mu}{z_\mu}.
    \end{equation}
    Let us say that a function $\mathbf{d} \colon [\ell] \to \N$ is \emph{permissible} if $\mathbf{d}(i) \mid \lambda_i$ for all~$i$. The product on the right-hand side of \eqref{equation:powersumproduct} can be expanded into a sum over all permissible functions:
    \begin{align}
        \F{p_\lambda} &= \sum_{\mu} \sum_{\text{$\mathbf{d}$ permissible}} \left(\prod_{i=1}^\ell \mathbf{d}(i) m_{\mathbf{d}(i)}(\mu)\right) \frac{p_\mu}{z_\mu} \nonumber\\
        &= \sum_{\mu} \sum_{\text{$\mathbf{d}$ permissible}} \left(\prod_{d} (d m_d(\mu))^{|\mathbf{d}^{-1}(d)|}\right) \frac{p_\mu}{z_\mu}.\label{equation:mdmu}
    \end{align}

    For any~$k$, let $(x)_{k} = x(x-1)\cdots(x-k+1)$ denote the falling factorial. It is well-known \cite[Chapter~6.1]{MR1397498} that for any~$n$, the monomial $x^n$ can be written as a linear combination
    \[x^n  = \sum_{k} \stirlingii{n}{k} (x)_k\]
    of falling factorials, where the coefficient $\stirlingii{n}{k}$ (a \emph{Stirling number of the second kind}) is the number of partitions of $[n]$ into~$k$ nonempty sets. Let us use this to rewrite the factor $(m_d(\mu))^{{|\mathbf{d}^{-1}(d)|}}$ appearing in \eqref{equation:mdmu}. We obtain
    \begin{equation}\label{equation:stirlingiiplambda}
        \F{p_\lambda} = \sum_{\mu} \sum_{\text{$\mathbf{d}$ permissible}} \left(\prod_{d} d^{|\mathbf{d}^{-1}(d)|} \sum_{k} \stirlingii{|\mathbf{d}^{-1}(d)|}{k} (m_d(\mu))_{k}\right) \frac{p_\mu}{z_\mu}.
    \end{equation}

    Given a permissible function $\mathbf{d} \colon [\ell] \to \N$ and a partition~$\pi$ of $[\ell]$ into nonempty sets, let us say that~$\pi$ is \emph{level} with respect to $\mathbf{d}$ if $\mathbf{d}(i) = \mathbf{d}(j)$ whenever~$i$ and~$j$ are in the same part of~$\pi$. If~$\pi$ is level with respect to $\mathbf{d}$, we may define the function $\widetilde{\mathbf{d}} \colon \pi \to \N$ by taking $\widetilde{\mathbf{d}}(U)$ to be the common value of $\mathbf{d}(i)$ for $i \in U$.

    Suppose that $\mathbf{d}$ is a fixed permissible function and $\{k_d\}_d$ is any sequence. Then it is easy to see that the product \[\prod_{d} \stirlingii{|\mathbf{d}^{-1}(d)|}{k_d}\] is equal to the number of set partitions~$\pi$ that are level with respect to~$U$ and which satisfy $|\widetilde{\mathbf{d}}^{-1}(d)| = k_d$ for all~$d$. Using this fact, we may expand the product in \eqref{equation:stirlingiiplambda} into a sum over all functions that are level with respect to $\mathbf{d}$:
    \[
        \F{p_\lambda} = \sum_{\mu} \sum_{\text{$\mathbf{d}$ permissible}} \sum_{\text{$\pi$ level}} \left(\prod_{d} d^{|\mathbf{d}^{-1}(d)|} (m_d(\mu))_{|\widetilde{\mathbf{d}}^{-1}(d)|}\right) \frac{p_\mu}{z_\mu}
    \]
    We will simplify this expression by switching the order of summation so that the sum over~$\pi$ is all the way on the outside. To do so, given a partition~$\pi$ of $[\ell]$ into nonempty sets, we will now describe the set of all permissible functions $\mathbf{d} \colon [\ell] \to \N$ such that~$\pi$ is level with respect to $\mathbf{d}$. These are exactly the functions given by $\mathbf{d}(i) = \widetilde{\mathbf{d}}(U)$ for all $U \in \pi$ and $i \in U$, where $\widetilde{\mathbf{d}} \colon \pi \to \N$ is any function satisfying \[\widetilde{\mathbf{d}}(U) \mid \gcd\{\lambda_i \colon i \in U\}\] for all $U \in \pi$. Let us call such functions \emph{$\pi$-permissible}. Then
    \[
        \F{p_\lambda} = \sum_{\pi} \sum_{\text{$\widetilde{\mathbf{d}}$~$\pi$-permissible}}\sum_{\mu} \left(\prod_{d} d^{|\mathbf{d}^{-1}(d)|}  (m_d(\mu))_{|\widetilde{\mathbf{d}}^{-1}(d)|}\right) \frac{p_\mu}{z_\mu}.
    \]
    Now, let us evaluate the inner sum over~$\mu$. Expanding $p_\mu$ and $z_\mu$ gives
    \begin{align*}
        \F{p_\lambda} &= \sum_{\pi} \sum_{\text{$\widetilde{\mathbf{d}}$~$\pi$-permissible}}\sum_{\mu} \left(\prod_{d} d^{|\mathbf{d}^{-1}(d)|}  (m_d(\mu))_{|\widetilde{\mathbf{d}}^{-1}(d)|}\frac{p_d^{m_d(\mu)}}{d^{m_d(\mu)}(m_d(\mu))!}\right) \\
        &= \sum_{\pi} \sum_{\text{$\widetilde{\mathbf{d}}$~$\pi$-permissible}} \prod_{d}\left(d^{|\mathbf{d}^{-1}(d)|} \sum_{m = 0}^\infty  (m)_{|\widetilde{\mathbf{d}}^{-1}(d)|}\frac{p_d^{m}}{d^{m}m!}\right) \\
        &= \sum_{\pi} \sum_{\text{$\widetilde{\mathbf{d}}$~$\pi$-permissible}} \prod_{d}\left(d^{|\mathbf{d}^{-1}(d)|} \sum_{m = |\widetilde{\mathbf{d}}^{-1}(d)|}^\infty \frac{p_d^{m}}{d^{m}(m - |\widetilde{\mathbf{d}}^{-1}(d)|)!}\right) \\
        &= \sum_{\pi} \sum_{\text{$\widetilde{\mathbf{d}}$~$\pi$-permissible}} \prod_{d}\left(d^{|\mathbf{d}^{-1}(d)|} \left(\frac{p_d}{d}\right)^{|\widetilde{\mathbf{d}}^{-1}(d)|} \exp\left(\frac{p_d}{d}\right)\right) \\
        &= \left(\sum_{\pi} \sum_{\text{$\widetilde{\mathbf{d}}$~$\pi$-permissible}} \prod_{d}\left(d^{|\mathbf{d}^{-1}(d)|} \left(\frac{p_d}{d}\right)^{|\widetilde{\mathbf{d}}^{-1}(d)|}\right)\right) \cdot H \\
        &= \left(\sum_{\pi} \sum_{\text{$\widetilde{\mathbf{d}}$~$\pi$-permissible}} \prod_{d}\left(d^{|\mathbf{d}^{-1}(d)| - |\widetilde{\mathbf{d}}^{-1}(d)|} \cdot p_d^{|\widetilde{\mathbf{d}}^{-1}(d)|}\right)\right) \cdot H.
    \end{align*}
    We will rewrite the remaining product as a product over $U \in \pi$ instead of over $d \in \N$. Each $U \in \pi$ with $\widetilde{\mathbf{d}}(U) = d$ contributes $1$ to $|\widetilde{\mathbf{d}}^{-1}(d)|$ and contributes $|U|$ to $|\mathbf{d}^{-1}(d)|$. So we obtain
    \begin{align*}
        \F{p_\lambda} = \left(\sum_{\pi} \sum_{\text{$\widetilde{\mathbf{d}}$~$\pi$-permissible}} \prod_{U \in \pi}(\widetilde{\mathbf{d}}(U))^{|U| - 1}p_{\widetilde{\mathbf{d}}(U)}\right) \cdot H.
    \end{align*}
    Given that $\widetilde{\mathbf{d}}$ is~$\pi$-permissible, the possible values of $\widetilde{\mathbf{d}}(U)$ are exactly those $d \in \N$ that divide $\gcd\{\lambda_i \colon i \in U\}$. The choice of $\widetilde{\mathbf{d}}(U)$ can be made independently for each $U \in \pi$. Hence, we may factor the innermost sum, which finally yields
    \[
        \F{p_\lambda} = \left(\sum_{\pi} \prod_{U \in \pi}\left(\sum_{d \mid \gcd\{\lambda_i \colon i \in U\}} d^{|U| - 1}p_{d}\right)\right) \cdot H.
    \]
    The result follows from dividing both sides of this equation by~$H$.
\end{proof}
\begin{rema}
    \Cref{theorem:computation}(a) has an alternate, ``purely combinatorial'' proof. It involves exhibiting a combinatorial species $E_\lambda \colon \rmBij \to \rmFun$ with $\ch(\K E_\lambda) = h_\lambda$, and then computing $\iota^* \iota_! E_\lambda$ in terms of known combinatorial species. The details for this proof are forthcoming in a separate paper.
\end{rema}
\begin{rema}
One consequence of \Cref{theorem:computation}(c) which is not obvious \emph{a priori} is that the matrix entries of $\Fa_{\Sur}$ in the power sum basis are all nonnegative integers. The same is not true of~$\Fa$; for example, $\F{1} = H = 1 + p_1 + \frac{1}{2} (p_2 + p_{1}^2) + \cdots$ certainly has some non-integer coefficients.
\end{rema}

To illustrate the utility of \Cref{theorem:computation}, we now restate and prove \Cref{theorem:durfeebound} about the vanishing of restriction coefficients. Recall from the introduction that $D(\mu)$ is the size of the Durfee square of~$\mu$; that is, the largest integer~$d$ such that $\mu_d \geq d$.

\durfeebound*

Before proceeding to the proof, we need a few lemmas.
\begin{lemma}\label{lemma:spans}
    For any $k \geq 0$, we have \[\operatorname{span} \{e_\lambda \colon \ell(\lambda) \leq k\} = \operatorname{span} \{s_\lambda \colon \lambda_1 \leq k\}.\] \textup{(}Here $\operatorname{span} S$ refers to the additive subgroup of~$\Lambda$ generated by a subset $S \subseteq \Lambda$.\textup{)}
\end{lemma}
\begin{proof}
    By the Pieri rule, every $e_\lambda$ with $\ell(\lambda) \leq k$ is a linear combination of Schur functions $s_\mu$, where~$\mu$ is the union of at most~$k$ vertical strips. Hence,
    \[\operatorname{span} \{e_\lambda \colon \ell(\lambda) \leq k\} \subseteq \operatorname{span} \{s_\lambda \colon \lambda_1 \leq k\}.\]

    By the dual Jacobi-Trudi identity, every $s_\lambda$ with $\lambda_1 \leq k$ can be written as the determinant of a $k \times k$ matrix whose entries are elementary symmetric functions $e_r$. Hence,
    \[\operatorname{span}\{s_\lambda \colon \lambda_1 \leq k\} \subseteq \operatorname{span} \{e_\lambda \colon \ell(\lambda) \leq k\}.\] The result follows.
\end{proof}
\begin{lemma}\label{lemma:durfee}
    Let~$\lambda$ be a partition and let $d \geq 0$. Then $D(\lambda) \leq d$ if and only if there exist partitions~$\mu$,~$\nu$ with $\ell(\mu), \ell(\nu) \leq d$ such that $s_\lambda$ appears in $h_{\mu} e_{\nu}$.
\end{lemma}
\begin{proof}
    For the ``only if'' direction, assume that $D(\lambda) \leq d$. Take
    \begin{align*}
        \mu &= (\lambda_1, \ldots, \lambda_{D(\lambda)}) \\
        \nu &= (\lambda^T_1 - D(\lambda), \ldots, \lambda^T_{D(\lambda)} - D(\lambda)).
    \end{align*} Clearly $\ell(\mu), \ell(\nu) \leq d$. Also, the Young diagram of~$\lambda$ can be decomposed into a union of horizontal strips of lengths $\mu_1, \ldots, \mu_{\ell(\mu)}$ and vertical strips of lengths $\nu_1, \ldots, \nu_{\ell(\nu)}$, so $s_\lambda$ appears in $h_\mu e_\nu$ by the Pieri rule.

    For the ``if'' direction, it easily follows from induction on $\ell(\nu)$ and the Pieri rule that if $s_\lambda$ appears in $h_\mu e_\nu$, then $\lambda_{\ell(\mu) + 1} < \ell(\nu) + 1$ (where we take $\lambda_i = 0$ for $i > \ell(\lambda)$). Assuming that $\ell(\mu), \ell(\nu) \leq d$, we get $\lambda_{d + 1} < d + 1$, so $D(\lambda) \leq d$ as desired.
\end{proof}

\begin{proof}[Proof of \Cref{theorem:durfeebound}]
    First, observe that by multiplying both sides by~$H$, \Cref{theorem:computation}(b) can be written in the following form, using the Frobenius transform instead of the surjective Frobenius transform. For any partition~$\lambda$ with $\ell(\lambda) \leq k$, we have
    \begin{equation}\label{equation:felambda}
        \F{e_\lambda} = \sum_{M} \prod_{j \in \{0, 1\}^k}
        \begin{cases}
            h_{M(j)} & \mbox{if $j_1 + \cdots + j_\ell$ is even;} \\
            e_{M(j)} & \mbox{if $j_1 + \cdots + j_\ell$ is odd,}
        \end{cases}
    \end{equation}
    where the sum is over all functions $M \colon \{0, 1\}^k \to \N$ such that $\sum_{j \in \{0, 1\}^k} j_i M(j) = \lambda_i$ for $i = 1, \ldots, k$ (where we take $\lambda_i = 0$ for $i > \ell(\lambda)$). We do not require that $M(0, \ldots, 0) = 0$, so the sum in \eqref{equation:felambda} is infinite.

    Consider the following statements:
    \begin{itemize}
        \item[(A)] There exists a partition~$\lambda$ such that $\lambda_1 \leq k$ and $r_{\lambda}^\mu > 0$.
        \item[(C\textsubscript{1})] There exists a partition~$\lambda$ such that $\lambda_1 \leq k$ and $\langle \F{s_{\lambda}}, s_\mu \rangle \neq 0$.
        \item[(C\textsubscript{2})] There exists a partition~$\lambda$ such that $\ell(\lambda) \leq k$ and $\langle \F{e_{\lambda}}, s_\mu \rangle \neq 0$.
        \item[(C\textsubscript{3})] There exists a function $M \colon \{0, 1\}^k \to \N$ such that $s_\mu$ appears in \[\prod_{j \in \{0, 1\}^k} \begin{cases} h_{M(j)} & \mbox{if $j_1 + \cdots + j_\ell$ is even;} \\
        e_{M(j)} & \mbox{if $j_1 + \cdots + j_\ell$ is odd.} \end{cases}\]
        \item[(C\textsubscript{4})] There exist partitions $\nu, \nu'$ such that $\ell(\nu), \ell(\nu') \leq 2^{k - 1}$ and $s_\mu$ appears in $h_\nu e_{\nu'}$.
        \item[(B)] $D(\mu) \leq 2^{k - 1}$.
    \end{itemize}
    By the definition of the Frobenius transform, (A) is equivalent to (C\textsubscript{1}). By the linearity of~$\Fa$ and \Cref{lemma:spans}, (C\textsubscript{1}) is equivalent to (C\textsubscript{2}). Since each term of \eqref{equation:felambda} is Schur positive, $s_\mu$ appears in the sum if and only if it appears in one or more of its terms, so (C\textsubscript{2}) is equivalent to (C\textsubscript{3}). By taking the parts of~$\nu$ to be the nonzero values of $M(j)$ for $j_1 + \cdots + j_\ell$ even and taking the parts of $\nu'$ to be the nonzero values of $M(j)$ for $j_1 + \cdots + j_\ell$ odd, we see that (C\textsubscript{3}) is equivalent to (C\textsubscript{4}). By \Cref{lemma:durfee}, (C\textsubscript{4}) is equivalent to (B). Putting it all together, (A) is equivalent to (B), as desired.
\end{proof}

\section{Computations of the inverse surjective Frobenius transform}
\label{section:invcomputations}
We will now compute $\FSurInv{e_\lambda}$ and $\FSurInv{h_\lambda}$. In order to state our formulas, first we must recall some definitions from combinatorics on words. For a more complete introduction, see \cite[Chapter~5]{MR1475463}.

\begin{defi}
    Let~$A$ be a set. A \emph{word} over the alphabet~$A$ is a sequence $w = w_1 \cdots w_n$ with $w_1, \ldots, w_n \in A$. Given a letter $a \in A$, we write $m_a(w)$ to denote the number of times the letter~$a$ appears in~$w$.
\end{defi}
\begin{defi}[\cite{MR0067884}]
    Let~$A$ be a totally ordered set. We say that a nonempty word $w = w_1 \cdots w_n$ over the alphabet~$A$ is a \emph{Lyndon word} if it is lexicographically less than its suffix $w_i \cdots w_n$ for $i = 2, \ldots, n$. Let $\Lyndon(A)$ be the set of all Lyndon words over the alphabet~$A$.
\end{defi}
\begin{theorem}[{Witt's Formula \cite{MR2035110}}]\label{theorem:witt}
    Let $\ell > 0$ and let $t_1, \ldots, t_\ell$ be variables. For any word $w = i_1 \cdots i_n$ over $[\ell]$, denote by $t^w$ the product \[t_{i_1} \cdots t_{i_n} = \prod_{i=1}^\ell t_i^{m_i(w)}.\] Then the evaluation $(L_1 + L_2 + L_3 + \cdots)(t_1, \ldots, t_\ell)$ is equal to \[\sum_{w \in \Lyndon([\ell])} t^w.\]
\end{theorem}
\begin{theorem}[{Chen--Fox--Lyndon Theorem \cite{MR0102539}}]\label{theorem:chenfoxlyndon}
    Let~$A$ be a totally ordered set. Any word~$w$ over the alphabet~$A$ has a unique \emph{Lyndon factorization}; that is, an expression as a (lexicographically) non-increasing concatenation of Lyndon words.
\end{theorem}
\begin{defi}
    Let~$w$ be a word over a totally ordered alphabet. Define $\pi(w)$ to be the partition obtained by listing the number of times each Lyndon word appears in the Lyndon factorization of~$w$, and then sorting the resulting positive numbers in decreasing order.
\end{defi}
\begin{exam}
    If $A = \{1, 2\}$ and $w = 21212121111$, then the Lyndon factorization of~$w$ is $w = (2)(12)(12)(12)(1)(1)(1)(1)$. The Lyndon words appearing in this factorization are $2$, $12$, and $1$, which appear once, three times, and four times, respectively, so $\pi(w) = (4, 3, 1)$.
\end{exam}

Now, we are ready to compute $\FSurInv{e_\lambda}$ and $\FSurInv{h_\lambda}$. The cleanest way to state our results is to use the series $H(t)$ from \Cref{section:preliminaries}.
\begin{theorem}
\label{theorem:invcomputation}
Let $\ell > 0$ and let $t_1, \ldots, t_\ell$ be variables.
\begin{enumerate}[label=(\alph*)]
    \item For any word $w = i_1 \cdots i_n$ over $[\ell]$, denote by $t^w$ the product \[t_{i_1} \cdots t_{i_n} = \prod_{i=1}^\ell t_i^{m_i(w)}.\]Then \[\FSurInv{\frac{1}{\prod_{i = 1}^\ell H(t_i)}} = \frac{1}{\prod_{w \in \Lyndon([\ell])} H(t^w)}.\]
    \item For any word $w = (i_1, j_1) \cdots (i_n, j_n)$ over $[\ell]^2$ (ordered lexicographically), denote by $t^w$ the product \[(t_{i_1} t_{j_1}) \cdots (t_{i_n} t_{j_n}) = \prod_{i, j = 1}^\ell (t_i t_j)^{m_{(i, j)}(w)}.\] Then \[\FSurInv{\prod_{i = 1}^\ell H(t_i)} = \frac{\prod_{w \in \Lyndon([\ell])} H(t^w)}{\prod_{w \in \Lyndon([\ell]^2)} H(t^w)}.\]
\end{enumerate}
\end{theorem}
Before we proceed to the proof of \Cref{theorem:invcomputation}, we will prove some corollaries that illustrate how to use it.
\begin{coro}
    For any~$r$, we have \[\FSurInv{h_r} = \sum_{k = 0}^{\lfloor r/2 \rfloor} (-1)^k h_{r - 2k} e_k.\]
\end{coro}
\begin{proof}
    Take $\ell = 1$ in \Cref{theorem:invcomputation}(b). In this case, $1$ is the only Lyndon word over $[1]$ and $(1, 1)$ is the only Lyndon word over $[1]^2$. So, with $t = t_1$,
    \[\FSurInv{H(t)} = \frac{H(t)}{H(t^2)} = H(t) E(-t^2).\] The result follows from taking the coefficient of $t^r$.
\end{proof}
\begin{coro}\label{corollary:fsurinve}
    Let $\lambda = (\lambda_1, \ldots, \lambda_\ell)$ be a sequence of nonnegative integers \textup{(}not necessarily weakly decreasing\textup{)}. Then \[\FSurInv{e_\lambda} = \sum_{w \in W} (-1)^{|\lambda| - |\pi(w)|} e_{\pi(w)},\] where~$W$ is the set of all words~$w$ over $[\ell]$ such that $m_i(w) = \lambda_i$ for all $i \in [\ell]$.
\end{coro}
\begin{proof}
    First, the right-hand side of \Cref{theorem:computation}(a) can be written
    \begin{equation}\label{equation:lyndonfunctions}
        \prod_{w \in \Lyndon(\ell)} \left(\sum_{r = 0}^\infty (-1)^r e_r (t^w)^r\right) = \sum_{\mathbf{r}} \prod_{w \in \Lyndon(\ell)} (-1)^{\mathbf{r}(w)} e_{\mathbf{r}(w)} (t^w)^{\mathbf{r}(w)}
    \end{equation}
    where the sum is over all finitely supported functions $\mathbf{r} \colon \Lyndon([\ell]) \to \N$. By the Chen--Fox--Lyndon theorem (\Cref{theorem:chenfoxlyndon}), there is a bijection \[\{\text{words over $[\ell]$}\} \underset{\phi}{\longleftrightarrow} \{\text{finitely supported functions $\Lyndon([\ell]) \to \N$}\},\] where $(\phi(w))(w')$ is the number of times that $w'$ appears in the Lyndon factorization of~$w$. Using this bijection, we can rewrite \eqref{equation:lyndonfunctions} as a sum over all words~$w$ over $[\ell]$:
    \[\FSurInv{\frac{1}{\prod_{i = 1}^\ell H(t_i)}} = \sum_{w} (-1)^{|\pi(w)|} e_{\pi(w)} t^w.\]
    The result follows from taking the coefficient of $t_1^{\lambda_1} \cdots t_{\ell}^{\lambda_{\ell}}$.
\end{proof}
\begin{coro}\label{corollary:e1ell}
    Let $\ell \geq 0$. Then \[\FSurInv{e_1^\ell} = e_1(e_1 - 1) \cdots (e_1 - \ell + 1).\]
\end{coro}
\begin{proof}
    Take $\lambda = (1^\ell)$ in \Cref{corollary:fsurinve}. Then $W = \Sym_\ell$. Moreover, the Lyndon factorization of any $w \in \Sym_\ell$ contains only distinct Lyndon words, one beginning with each left-to-right minimum of~$w$ (that is, each letter $w_i$ such that $w_i < w_j$ for all $j < i$). So $\pi(w) = (1^k)$, where~$k$ is the number of left-to-right minima of~$w$. It follows that \[\FSurInv{e_1^\ell} = \sum_{k} (-1)^{\ell - k}\stirlingi{\ell}{k} e_1^k,\] where $\stirlingi{\ell}{k}$ (a \emph{Stirling number of the first kind}) is the number of permutations $w \in \Sym_\ell$ with~$k$ left-to-right minima. This is well-known \cite[Chapter~6.1]{MR1397498} to be equal to $e_1(e_1 - 1) \cdots (e_1 - \ell + 1)$.
\end{proof}
More generally, we have the following.
\begin{coro}\label{corollary:fsurinveprod}
    Let $\lambda = (\lambda_1, \ldots, \lambda_\ell)$ be a sequence of nonnegative integers \textup{(}not necessarily weakly decreasing\textup{)} and let $k \geq 0$. For any word~$w$ over $\{0\} \cup [\ell]$, write $p_+(w)$ to denote the longest prefix of~$w$ that does not contain the letter $0$. Then \[\FSurInv{e_\lambda e_1^k} = \left(\sum_{w \in W} (-1)^{|\lambda| - |\pi(p_+(w))|} e_{\pi(p_+(w))} \right) \cdot e_1 (e_1 - 1) \cdots (e_1 - k + 1),\] where~$W$ is the set of all words~$w$ over $\{0\} \cup [\ell]$ such that $m_0(w) = k$ and $m_i(w) = \lambda_i$ for all $i \in [\ell]$.
\end{coro}
\begin{proof}
    Let \[\lambda' = (\underbrace{1, \ldots, 1}_{k}, \lambda_1, \ldots, \lambda_\ell)\] and let $W'$ be the set of all words~$w$ over $[k + \ell]$ such that $m_i(w) = \lambda'_i$ for all $i \in [k + \ell]$. By \Cref{corollary:fsurinve}, we have
    \begin{equation}\label{eq:wprime}
        \FSurInv{e_\lambda e_1^k} = \sum_{w \in W'} (-1)^{|\lambda| + k - |\pi(w)|} e_{\pi(w)}.
    \end{equation}
    Now, define the function \[\phi \colon W' \to W \times \Sym_k\] as follows. For any word $w \in W'$, let $\phi_1(w) \in W$ be the word formed from~$w$ by replacing all the letters $1, \ldots, k$ with $0$ and replacing all copies of the letters $k + 1, \ldots, k + \ell$ with $1, \ldots, \ell$ respectively. Let $\phi_2(w) \in \Sym_k$ be the word formed from~$w$ by deleting all copies of the letters $k+1, \ldots, k+\ell$. It is easy to see that $\phi = (\phi_1, \phi_2)$ is a bijection. Moreover, for all $w \in W'$, we have that $\pi(w)$ is the concatenation of $\pi(p_+(\phi_1(w)))$ and $\pi(\phi_2(w))$. Hence, we may factor \eqref{eq:wprime}:
    \begin{align*}
        \FSurInv{e_\lambda e_1^k} = \left(\sum_{w \in W} (-1)^{|\lambda| - |\pi(p_+(w))|} e_{\pi(p_+(w))}\right) \left( \sum_{w \in \Sym_k} (-1)^{k - |\pi(w)|} e_{|\pi(w)|}\right).
    \end{align*}
    As in the proof of \Cref{corollary:e1ell}, the second factor is equal to $e_1(e_1 - 1) \cdots (e_1 - k + 1)$. The result follows.
\end{proof}
\begin{coro}
    Let $f \in \Lambda$ and $k \geq 0$. Then $e_1(e_1 - 1) \cdots (e_1 - k + 1)$ divides~$f$ if and only if $e_1^k$ divides $\FSur{f}$.
\end{coro}
\begin{proof}
    Let $n = \deg f$. Let~$I$ be the set of all symmetric functions of degree at most~$n$ that are divisible by $e_1(e_1 - 1) \cdots (e_1 - k + 1)$ and let~$J$ be the set of all symmetric functions of degree at most~$n$ that are divisible by $e_1^k$. Now,~$J$ is spanned by symmetric functions of the form $e_\lambda e_1^k$, where~$\lambda$ is a partition with $|\lambda| \leq n - k$. By \Cref{corollary:fsurinveprod}, we have $\FSurInv{J} \subseteq I$. Since~$I$ and~$J$ have the same dimension (and $\Lambda / J$ is torsion-free), it follows that $\FSurInv{J} = I$. Thus, $\FSur{I} = J$. The result follows.
\end{proof}
We are almost ready to prove \Cref{theorem:computation}. First, we will restate some lemmas from Loehr and Remmel's 2011 ``expos\'e'' on plethysm \cite{MR2765321}.
\begin{lemma}[{\cite[Example~1]{MR2765321}}]\label{lemma:plethysm}
    There is a unique binary operation $\bullet[\bullet] \colon \Lambda \times \Z\llbracket t_1, \cdots, t_\ell \rrbracket \to \Z\llbracket t_1, \cdots, t_\ell \rrbracket$ satisfying the following properties.
    \begin{enumerate}[label=(\roman*)]
        \item For any fixed $g \in \Z\llbracket t_1, \cdots, t_\ell \rrbracket$, the function $\bullet[g] \colon \Lambda \to \Z\llbracket t_1, \cdots, t_\ell \rrbracket$ is a ring homomorphism.
        \item For any fixed $k > 0$, the function $p_k[\bullet] \colon \Z\llbracket t_1, \cdots, t_\ell\rrbracket \to \Z\llbracket t_1, \cdots, t_\ell\rrbracket$ is a ring homomorphism which preserves summable infinite series.
        \item For any~$k$ and~$i$, we have $p_k[t_i] = t_i^k$.
    \end{enumerate}
\end{lemma}
We refer to the operation from \Cref{lemma:plethysm} as \emph{plethysm}, because it is closely related to the plethysm of symmetric functions $\bullet[\bullet] \colon \Lambda \times \Lambda \to \Lambda$ mentioned in \Cref{section:preliminaries}. For example, the two operations are related by a kind of associative property:
\begin{lemma}[{Associativity of Plethysm, \cite[Theorem~5]{MR2765321}}]\label{lemma:associativity}
    Let $f, g \in \Lambda$ and $h \in \Z\llbracket t_1, \cdots, t_\ell \rrbracket$. Then \[f[g[h]] = (f[g])[h] \in \Z\llbracket t_1, \cdots, t_\ell \rrbracket.\]
\end{lemma}
If~$g$ has positive integer coefficients, then the plethysm $f[g]$ can be described as an evaluation:
\begin{lemma}[{Monomial Substitution Rule, \cite[Theorem~7]{MR2765321}}]\label{lemma:monomialsubstitution}
    Let $f \in \Lambda$ and let $M_1, M_2, M_3, \ldots \in \Z\llbracket t_1, \cdots, t_\ell\rrbracket$ be a \textup{(}finite or infinite\textup{)} sequence of monic monomials. Then
    \[f\left[\sum_n M_n\right] = f(M_1, M_2, M_3, \ldots),\] where the right-hand side denotes the evaluation of~$f$ at $M_1, M_2, M_3, \ldots$.
\end{lemma}
In general, plethysm can be expressed as a Hall inner product:
\begin{lemma}\label{lemma:plethysmhall}
    Let $M_1, M_2, M_3, \ldots \in \Z[t_1, \cdots, t_\ell]$ be a \textup{(}finite or infinite\textup{)} sequence of monic monomials and let $a_1, a_2, a_3, \ldots \in \Z$ be a sequence of the same length. Suppose that the series $\sum_{n} a_n M_n$ is summable in $\Z\llbracket t_1, \cdots, t_\ell\rrbracket$. Then for any $f \in \Lambda$, we have
    \begin{equation}\label{equation:moniclemma}
        f\left[\sum_{n} a_n M_n\right] = \left\langle f, \prod_{n} H(M_n)^{a_n}\right\rangle.
    \end{equation}
\end{lemma}
\begin{proof}
    Fix $M_1, M_2, M_3, \ldots$ and let $a_1, a_2, a_3, \ldots$ vary. Given any fixed monomial $M \in \Z[t_1, \ldots, t_\ell]$, it is easy to see that the coefficient of~$M$ on either side of \eqref{equation:moniclemma} is a polynomial in $a_1, a_2, a_3, \ldots$. Hence, we may assume that $a_n \geq 0$ for all~$n$. Then by replacing each $M_n$ with $a_n$ copies of $M_n$, we may make the even stronger assumption that $a_n = 1$ for all~$n$. In other words, we wish to prove
    \begin{equation}\label{equation:innerproducth}
        f\left[\sum_{n} M_n\right] = \left\langle f, \prod_{n} H(M_n)\right\rangle.
    \end{equation}
    By linearity, it suffices to show \eqref{equation:innerproducth} in the case that $f = m_\lambda$ is a monomial symmetric function. Then, since $\langle m_\lambda, h_\mu \rangle = \delta_{\lambda \mu}$, the right-hand side of the equation becomes $m_\lambda(M_1, M_2, M_3, \ldots)$. The result follows from the monomial substitution rule (\Cref{lemma:monomialsubstitution}).
\end{proof}
\begin{proof}[Proof of \Cref{theorem:invcomputation}]
    In what follows, let \[L = L_1 + L_2 + L_3 + \cdots \in \overline{\Lambda}\] and \[\tilde{L} = \omega(L_1) - \omega(L_2) + \omega(L_3) - \cdots \in \overline{\Lambda}.\]

    \statement{(a)} Let $\overline{\omega} \colon \Lambda \to \Lambda$ be the involution given by \[\overline{\omega}(f) = f[-p_1] = (-1)^{\deg f} \omega(f)\] for all homogeneous $f \in \Lambda$. Clearly, $\overline{\omega}$ is a ring automorphism and \[\overline{\omega}(H(t)) = E(-t) = \frac{1}{H(t)}.\] We wish to show that
    \begin{equation}\label{equation:omegafainv}
        (\overline{\omega} \circ \Fa_{\Sur}^{-1} \circ \overline{\omega}) \left\{\prod_{i=1}^\ell H(t_i)\right\} = \prod_{w \in \Lyndon([\ell])} H(t^w).
    \end{equation}
    To do so, let $f \in \Lambda$ be arbitrary. It is sufficient to show that each side of \eqref{equation:omegafainv} has the same Hall inner product with~$f$. By \Cref{theorem:summary}(c), $\overline{\omega} \circ \Fa_{\Sur}^{-1} \circ \overline{\omega}$ is adjoint to plethysm by~$L$, so
    \begin{align*}
        \left\langle f, (\overline{\omega} \circ \Fa_{\Sur}^{-1} \circ \overline{\omega}) \left\{\prod_{i=1}^\ell H(t_i)\right\}\right\rangle &= \left\langle f[L], \prod_{i=1}^\ell H(t_i)\right\rangle \\
        &= (f[L])[t_1 + \cdots + t_\ell] \\
        &= f[L[t_1 + \cdots + t_\ell]] \\
        &= f[L(t_1, \ldots, t_\ell)] \\
        &= f\left[\sum_{w \in \Lyndon([\ell])} t^w \right] \\
        &= \left \langle f, \prod_{w \in \Lyndon([\ell])} H(t^w) \right \rangle,
    \end{align*}
    where we used \Cref{lemma:plethysmhall}, \Cref{lemma:associativity}, \Cref{lemma:monomialsubstitution}, and \Cref{theorem:witt}. This completes the proof of \eqref{equation:omegafainv} and of part (a) of the theorem.

    \statement{(b)} Again, let $f \in \Lambda$ be arbitrary. It is sufficient to show that each side of the equation has the same Hall inner product with~$f$. By \Cref{theorem:summary}(c), $\Fa_{\Sur}^{-1}$ is adjoint to plethysm by $\tilde{L}$, so
    \begin{align}
        \left\langle f, \FSurInv{\prod_{i=1}^\ell H(t_i)}\right\rangle &= \left\langle f[\tilde{L}], \prod_{i=1}^\ell H(t_i)\right\rangle\nonumber\\
        &= (f[\tilde{L}])[t_1 + \cdots + t_\ell]\nonumber\\
        &= f[\tilde{L}[t_1 + \cdots + t_\ell]]\label{eq:tildel}
    \end{align}
    Now, let us describe the monomials that appear in the plethysm $\tilde{L}[t_1 + \cdots + t_\ell]$. By the definition of Lyndon symmetric functions, we have
    \begin{align*}
        \tilde{L} &=\sum_{n} (-1)^{n - 1} \omega(L_n)\\
        &= \sum_{n} \frac{(-1)^{n-1}}{n} \sum_{d \mid n} \mu(d) \omega(p_d^{n/d}) \\
        &= \sum_{n} \frac{(-1)^{n-1}}{n} \sum_{d \mid n} \mu(d) (-1)^{(d-1)n/d} p_d^{n/d} \\
        &= \sum_{n} \frac{1}{n} \sum_{d \mid n} \mu(d) (-1)^{n/d - 1} p_d^{n/d}.
    \end{align*}
    Now, the term $(-1)^{n/d - 1}$ is equal to $-1$ if~$d$ divides $\frac{n}{2}$ and $1$ otherwise. Hence,
    \[
        \tilde{L} = \sum_{n} \frac{1}{n}\left(\sum_{d \mid n} \mu(d) p_d^{n/d} - 2\sum_{d \mid \frac{n}{2}} \mu(d) p_d^{n/d}\right)
    \]
    where the sum over $d \mid \frac{n}{2}$ is understood to be empty if~$n$ is odd. Splitting this into two sums and then performing the change of variables $\frac{n}{2} \to n$ in the second, we obtain
    \begin{align*}
        \tilde{L} &= \sum_{n} \frac{1}{n}\sum_{d \mid n} \mu(d) p_d^{n/d} - \sum_{n} \frac{2}{n}\sum_{d \mid \frac{n}{2}} \mu(d) p_d^{n/d} \\
        &= \sum_{n} \frac{1}{n}\sum_{d \mid n} \mu(d) p_d^{n/d} - \sum_{n} \frac{1}{n}\sum_{d \mid n} \mu(d) p_d^{2n/d} \\
        &= L - L[p_1^2].
    \end{align*}
    By \Cref{lemma:associativity} and \Cref{theorem:witt},
    \begin{align*}
        \tilde{L}[t_1 + \cdots + t_\ell] &= L[t_1 + \cdots + t_\ell] - (L[p_1^2])[t_1 + \cdots + t_\ell] \\
        &= L[t_1 + \cdots + t_\ell] - L[p_1^2[t_1 + \cdots + t_\ell]] \\
        &= L[t_1 + \cdots + t_\ell] - L[(t_1 + \cdots + t_\ell)^2] \\
        &= \sum_{w \in \Lyndon([\ell])} t^w - \sum_{w \in \Lyndon([\ell]^2)} t^w.
    \end{align*}
    Substituting into \eqref{eq:tildel} and using \Cref{lemma:plethysmhall} one last time, we finally obtain
    \begin{align*}
        \left\langle f, \FSurInv{\prod_{i=1}^\ell H(t_i)} \right\rangle &= f\left[\sum_{w \in \Lyndon([\ell])} t^w - \sum_{w \in \Lyndon([\ell]^2)} t^w\right] \\
        &= \left\langle f, \frac{\prod_{w \in \Lyndon([\ell])} H(t^w)}{\prod_{w \in \Lyndon([\ell]^2)} H(t^w)}\right\rangle,
    \end{align*}
    completing the proof.
\end{proof}

\bibliographystyle{plain}
\bibliography{main}
\end{document}